\documentclass{article}

\AtEndDocument{

\bigskip{\footnotesize%
Y.~Harpaz: \textsc{Radboud Universiteit Nijmegen, Institute for
Mathematics, Astrophysics, and Particle Physics, Heyendaalseweg 135, 6525
AJ Nijmegen, The Netherlands} \par  
  \textit{E-mail address:} \texttt{y.harpaz@math.ru.nl} \par

  \addvspace{\medskipamount}
M.~Prasma: \textsc{Radboud Universiteit Nijmegen, Institute for
Mathematics, Astrophysics, and Particle Physics, Heyendaalseweg 135, 6525
AJ Nijmegen, The Netherlands} \par
  \textit{E-mail address:}  \texttt{mtnprsm@gmail.com}
}}

\usepackage[nottoc,numbib]{tocbibind}
\usepackage[centertags]{amsmath}
\usepackage{hyperref}
\usepackage{amsfonts}

\usepackage{MnSymbol}

\usepackage{amsthm}

\usepackage{newlfont}
\usepackage{amscd}
\usepackage{amsmath}
\usepackage{enumerate}
\usepackage[all,2cell]{xy}
\UseAllTwocells
\input xy
\xyoption{2cell}
\xyoption{all}

\usepackage{verbatim}
\usepackage{eucal}

\newcommand{\Cof}{\mathcal{C}of}
\newcommand{\Fib}{\mathcal{F}ib}

\newcommand{\NN}{\mathbb{N}}
\newcommand{\GG}{\mathfrak{G}}

\renewcommand{\phi}{\varphi}

\def\B{\mathbb{B}}
\def\C{\mathcal{C}}

\def\G{\mathcal{G}}

\def\F{\mathcal{F}}
\def\SS{\mathcal{S}}
\def\M{\mathcal{M}}
\def\N{\mathcal{N}}
\def\W{\mathcal{W}}
\def\I{\mathcal{I}}
\def\SS{\mathcal{S}}

\def\P{\mathcal{P}}
\def\J{\mathcal{J}}
\def\R{\mathcal{R}}
\renewcommand{\L}{\mathcal{L}}
\def\U{\mathcal{U}}
\def\V{\mathcal{V}}
\def\W{\mathcal{W}}

\def\sig{{\sigma}}

\def\Om{{\Omega}}

\def\Del{{\Delta}}
\def\Sig{{\Sigma}}

\def\Lam{{\Lambda}}
\def\vphi{\varphi}
\def\l{\left}
\def\r{\right}

\newcommand{\HSwarrow}{\kern0.05ex\vcenter{\hbox{\Huge\ensuremath{\Swarrow}}}\kern0.05ex}
\newcommand{\hSwarrow}{\kern0.05ex\vcenter{\hbox{\huge\ensuremath{\Swarrow}}}\kern0.05ex}
\newcommand{\LLSwarrow}{\kern0.05ex\vcenter{\hbox{\LARGE\ensuremath{\Swarrow}}}\kern0.05ex}
\newcommand{\LSwarrow}{\kern0.05ex\vcenter{\hbox{\Large\ensuremath{\Swarrow}}}\kern0.05ex}

\newcommand{\HSearrow}{\kern0.05ex\vcenter{\hbox{\Huge\ensuremath{\Searrow}}}\kern0.05ex}
\newcommand{\hSearrow}{\kern0.05ex\vcenter{\hbox{\huge\ensuremath{\Searrow}}}\kern0.05ex}
\newcommand{\LLSearrow}{\kern0.05ex\vcenter{\hbox{\LARGE\ensuremath{\Searrow}}}\kern0.05ex}
\newcommand{\LSearrow}{\kern0.05ex\vcenter{\hbox{\Large\ensuremath{\Searrow}}}\kern0.05ex}

\newcommand{\HDownarrow}{\kern0.05ex\vcenter{\hbox{\Huge\ensuremath{\Downarrow}}}\kern0.05ex}
\newcommand{\hDownarrow}{\kern0.05ex\vcenter{\hbox{\huge\ensuremath{\Downarrow}}}\kern0.05ex}
\newcommand{\LLDownarrow}{\kern0.05ex\vcenter{\hbox{\LARGE\ensuremath{\Downarrow}}}\kern0.05ex}
\newcommand{\LDownarrow}{\kern0.05ex\vcenter{\hbox{\Large\ensuremath{\Downarrow}}}\kern0.05ex}

\newcommand{\HUparrow}{\kern0.05ex\vcenter{\hbox{\Huge\ensuremath{\Uparrow}}}\kern0.05ex}
\newcommand{\hUparrow}{\kern0.05ex\vcenter{\hbox{\huge\ensuremath{\Uparrow}}}\kern0.05ex}
\newcommand{\LLUparrow}{\kern0.05ex\vcenter{\hbox{\LARGE\ensuremath{\Uparrow}}}\kern0.05ex}
\newcommand{\LUparrow}{\kern0.05ex\vcenter{\hbox{\Large\ensuremath{\Uparrow}}}\kern0.05ex}

\newtheorem{thm}{Theorem}[subsection]
\newtheorem{cor}[thm]{Corollary}
\newtheorem{lem}[thm]{Lemma}
\newtheorem{pro}[thm]{Proposition}

\theoremstyle{definition}
\newtheorem{define}[thm]{Definition}

\newtheorem{example}[thm]{Example}
\newtheorem{defn}[thm]{Definition}

\newtheorem{notn}[thm]{Notation}

\theoremstyle{remark}
\newtheorem{rem}[thm]{Remark}

\numberwithin{equation}{subsection}

\DeclareMathOperator{\holim}{holim}

\DeclareMathOperator{\Id}{Id}
\DeclareFontFamily{OT1}{pzc}{}
\DeclareFontShape{OT1}{pzc}{m}{it}{<-> s * [1.10] pzcmi7t}{}
\DeclareMathAlphabet{\mathpzc}{OT1}{pzc}{m}{it}
\DeclareMathOperator{\cof}{cof}
\DeclareMathOperator{\fib}{fib}

\DeclareMathOperator{\res}{res}

\DeclareMathOperator{\ModCat}{ModCat}
\DeclareMathOperator{\sGr}{sGr}
\DeclareMathOperator{\sS}{s\mathcal{S}}
\DeclareMathOperator{\diag}{diag}
\DeclareMathOperator{\op}{op}
\DeclareMathOperator{\Map}{Map}
\DeclareMathOperator{\red}{red}
\DeclareMathOperator{\inj}{inj}
\DeclareMathOperator{\df}{def}
\DeclareMathOperator{\Set}{Set}

\DeclareMathOperator{\Cat}{Cat}

\DeclareMathOperator{\Fun}{Fun}
\DeclareMathOperator{\AdjCat}{AdjCat}
\DeclareMathOperator{\Obj}{Obj}

\DeclareMathOperator{\P_n}{P_n}
\DeclareMathOperator{\racts}{\curvearrowright}

\DeclareMathOperator{\der}{h}
\DeclareMathOperator{\cosk}{cosk}
\DeclareMathOperator{\seg}{seg}
\DeclareMathOperator{\Seg}{\left(\sS_0\right)_{\seg}}
\DeclareMathOperator{\Act}{act}
\DeclareMathOperator{\sk}{sk}
\DeclareMathOperator{\equ}{equ}
\DeclareMathOperator{\WW}{W}
\DeclareMathOperator{\Mon}{Mon}

\DeclareMathOperator{\sMon}{sMon}
\DeclareMathOperator{\sSet}{sSet}
\DeclareMathOperator{\sGp}{sGp}
\DeclareMathOperator{\act}{act}
\DeclareMathOperator{\segal}{seg}

\DeclareMathOperator{\Sp}{Sp}

\def\x{\overset}

\newcommand{\tgpd}{\kern0.05ex\vcenter{\hbox{\footnotesize\ensuremath{2}}}\kern0.05ex\mathcal{G}pd} 

\def\rar{\rightarrow}

\def\lrar{\longrightarrow}

\def\hrar{\hookrightarrow}

\def\ovl{\overline}

\newcommand\ackname{Acknowledgements:}
\if@titlepage
  \newenvironment{acknowledgements}{%
      \titlepage
      \null\vfil
      \@beginparpenalty\@lowpenalty
      \begin{center}
        \bfseries \ackname\
        \@endparpenalty\@M
      \end{center}}%
     {\par\vfil\null\endtitlepage}
\else
  
\fi

\title{An integral model structure and truncation theory for coherent group actions}

\author{Yonatan Harpaz \;\;\;\; Matan Prasma}

\date{}
\begin{document}
\maketitle

\begin{abstract}
In this work we study the homotopy theory of coherent group actions from a global point of view, where we allow both the group and the space acted upon to vary. Using the model of Segal group actions and the model categorical Grothendieck construction we construct a model category encompassing all Segal group actions simultaneously. We then prove a global rectification result in this setting. We proceed to develop a general truncation theory for the model-categorical Grothendieck construction and apply it to the case of Segal group actions. We give a simple characterization of $n$-truncated Segal group actions and show that every Segal group action admits a convergent Postnikov tower.
%
\end{abstract}


\tableofcontents

\section{Introduction} 

Let $\SS$ be the category of simplicial sets (which we shall refer to as \textbf{spaces}) and let $G$ be a simplicial group. Homotopy theories of spaces equipped with an action of $G$ are of fundamental interest in algebraic topology. The robust machinery of modern homotopy theory can be applied to this theory in various ways. For example, one can form the \textbf{Borel model category} $\SS^{\B G}$ (where $\B G$ is the simplicial groupoid with one object and automorphism space $G$), whose objects are spaces equipped with an action of $G$ and whose weak equivalences are the $G$-equivariant maps which induce a weak equivalence on the underlying spaces. The category $\SS^{\B G}$ is a strict model for this theory -- each object in $\SS^{\B G}$ corresponds to an actual space equipped with an honest $G$-action. Alternatively, if one wishes to work in a non-strict setting,  one can replace $G$ with its classifying space $\ovl{W}(G)$ and use the model category $\SS_{/\ovl{W}(G)}$ of spaces over $\ovl{W}(G)$ (see~\cite{DDK})

In algebraic topology one often wishes to study group actions for several different groups simultaneously. This setup should combine the homotopy theory of simplicial groups and the homotopy theory of the various equivariant spaces in a compatible way. In a previous paper~\cite{HP}, the authors provide a general machinery to study this setup through a suitable model category. More specifically, given a family $\F$ of model categories parametrized by a model category $\M$, the authors construct a model structure on the Grothendieck construction $\int_\M\F \lrar \M$ called the \textbf{integral model structure}. This model structure combines the model structures of $\M$ and the various fibers $\F(A)$ in a coherent fashion and provides a model for the corresponding $\infty$-categorical Grothendieck construction. Moreover, the integral model structure is shown to be invariant (up to Quillen equivalence) under replacing the pair $(\M,\F)$ with a suitably Quillen equivalent one.

In~\cite{HP}, the authors provide two examples of integral model structures in the setting of group actions. In the first model one integrates the functor which assigns to each simplicial group $G$ the Borel model category $\SS^{\B G}$. In the second model one replaces simplicial groups with the equivalent model of reduced simplicial sets, under which a simplicial group $G$ is modelled by its classifying space $\ovl{W}(G)$. One can then integrate the functor which assigns to each reduced simplicial set $B$ the model category of spaces over $B$. The invariance of the integral model structure alluded to above is then used to show that the resulting model categories are Quillen equivalent. We refer to such integral model categories as \textbf{global homotopy theories} for group actions.

The two models described above both have advantages and disadvantages. In the strict case one has direct access to the simplicial group $G$ and the $G$-space $X$. However, working only with strict models makes it difficult to form homotopical constructions which only preserve Cartesian products up to homotopy. On the other hand, the weak model of spaces over $\ovl{W}(G)$ is flexible and amenable to homotopical constructions, but does not give a direct access to the actual group $G$ or the underlying space on which it acts. 

There exists a third model for groups and group actions which enjoys the advantages of both worlds. A famous result of Segal (see \cite[Proposition 1.5]{Seg}) essentially shows that the homotopy theory of simplicial groups is equivalent to the homotopy theory of \textbf{Segal groups} (called \textbf{special $\Del$-spaces} in~\cite{Seg}), a model in which the group structure is encoded in a homotopy coherent way. It is tempting to try and extend the correspondence between simplicial groups and Segal groups into group actions. This was indeed done in \cite{Pra} where the author defines the notion of a \textbf{Segal group action} over a fixed Segal group $A_\bullet$ and constructs a model category for such objects. This model category is then shown to be Quillen equivalent to the Borel model category $\SS^{\B G}$, where $G$ is a simplicial group model for $A_\bullet$.

In this paper, we will adapt the construction of~\cite{Pra} to the setting of the integral model structure. This will require setting up a good model category for Segal groups, and showing that the functor which associates to each Segal group its model category of Segal group actions can be integrated in the sense discussed above. We will then prove that the resulting integral model category is equivalent to the two models constructed in~\cite{HP}. This can be viewed as a global rectification result for Segal group actions.

Finally, we study truncation theory in integral model categories and characterize $n$-truncated objects and $n$-truncation maps in terms of their counterparts in the base and in the fibers. We apply these results to case of Segal group actions where truncations are shown to take a particularly nice form. We will show that every Segal group action admits a convergent Postnikov tower. This gives a canonical filtration on any Segal group action, which can be used to obtain a corresponding filtration on strict group actions as well. We expect these constructions to have interesting applications in the future.


%
%
%

\section{Preliminaries }\label{s:prelim}

Throughout, a \textbf{space} will always mean a simplicial set and we denote by $\SS$ the category of spaces.

\subsection{Model categories}\label{ss:model}

We shall use an adjusted version of Quillen's original definition of a (closed) model category (see~\cite{Qui}).

\begin{defn}
A \textbf{model category} $\M$ is category with three distinguished classes of morphisms 
$\W=\W_{\M},\F ib=\F ib_{\M}$ and $\C of=\C of_{\M}\;$ called \textbf{weak equivalences}, \textbf{fibrations} and \textbf{cofibrations} (respectively), satisfying the following axioms:
\begin{enumerate}[MC1]
\item  The category $\M$ is complete and cocomplete.
\item  Each of the classes $\W,\Fib$ and $\Cof$ contains all isomorphisms and is closed under composition and retracts.
\item  If $f,g$ are composable maps such that two of $f,g$ and $gf$ are in $\W$ then so is the third. 
\item  Given the commutative solid diagram in $\M$ 
$$\xymatrix{A\ar[d]_{i}\ar[r] & X\ar[d]^p\\ B\ar[r]\ar@{-->}[ur] & Y\\}$$ in which $i\in \Cof$ and $p\in \Fib$, a dashed arrow exists if either $i$ or $p$ are in $\W$.
\item Any map $f$ in $\M$ has two functorial factorizations:
\begin{enumerate}[(i)]
\item $f=pi$ with $i\in \Cof$ and $p\in \Fib\cap \W$;
\item $f=qj$ with $j\in \Cof \cap \W$ and $q\in \Fib$.
\end{enumerate}
\end{enumerate} 

We shall refer to the maps in $\Fib\cap \W$ (resp. $\Cof\cap \W$) as \textbf{trivial fibrations} (resp. \textbf{trivial cofibrations}). For an object $X\in \M$ we denote by $X^{\fib}$ (resp. $X^{\cof}$) the functorial fibrant (resp. cofibrant) replacement of $X$, obtained by factorizing the map to the terminal object $X\lrar *$ (resp. from the initial object $\emptyset \lrar X$) into a trivial cofibration followed by a fibration (resp. a cofibration followed by trivial fibration). 
\end{defn}

\begin{defn}(cf. \cite[Definition $2.16$]{HP})
We denote by $\ModCat$ the $(2,1)$-category whose objects are the model categories and whose morphisms the Quillen pairs (composition is done in the direction of the left Quillen functor). The $2$-morphisms are given by the pseudo-natural isomorphisms of (Quillen) adjunctions.
\end{defn}

\subsection{The integral model structure}\label{ss:integral}

In this subsection we will recall the construction of the integral model structure on the Grothendieck construction of a diagram of model categories indexed by a model category (as developed in~\cite{HP}).
 
Suppose $\M$ is a model category and $\F:\M\rar \ModCat$ a (pseudo-)functor. For a morphism $f:A\lrar B$ in $\M$, we denote the associated adjunction in $\ModCat$ by
$$\xymatrix{f_{!}:\F(A)\ar[r]<1ex> & \F(B):f^*\ar[l]<1ex>_{\upvdash}.}$$
Recall that an object of the Grothendieck construction $\int_{\M} \F$ is a pair $(A,X)$ where $A\in \Obj\M$ and $X\in \Obj\F(A)$ and a morphism $(A,X)\rar (B,Y)$ in $\int_{\M} \F$ is a pair $(f,\phi)$ where $f:A\rar B$ is a morphism in $\M$ and $\phi:f_{!}X\rar Y$ is a morphism in $\F(B)$. In this case, we denote by $\phi^{ad}:X\rar f^*Y$ the adjoint map of $\phi$.
\begin{defn}\label{d:model}
Call a morphism $(f,\phi):(A,X)\rar (B,Y)$ in $\int_{\M} \F$
\begin{enumerate}
\item a \textbf{weak equivalence} if $f:A\rar B$ is a weak equivalence in $\M$ and the composite $f_{!}(X^{cof})\rar f_{!}X\rar Y$ is a weak equivalence in $\F(B)$;
\item a \textbf{fibration} if $f:A\rar B$ is a fibration and $\phi^{ad}:X\rar f^*Y$ is a fibration in $\F(A)$;
\item a \textbf{cofibration} if $f:A\rar B$ is a cofibration in $\M$ and $\phi:f_{!}X\rar Y$ is a cofibration in $\F(B)$.
\end{enumerate}
We denote these classes by $\W$, $\Fib$ and $\Cof$ respectively.
\end{defn}

\begin{defn}\label{d:relative}
We will say that a functor $\F: \M \lrar \ModCat$ is \textbf{relative} if for every weak equivalence $f:A\lrar B$ in $\M$, the associated Quillen pair $f_{!}\dashv f^*$ is a Quillen equivalence.
\end{defn}

\begin{defn}\label{d:proper}
Let $\M$ be a model category and $\F:\M\rar \ModCat$ a functor. We shall say that $\F$ is \textbf{proper} if whenever  $f:A \lrar B$ is a trivial cofibration in $\M$ the associated left Quillen functor preserves weak equivalences, i.e., $f_{!}(\W_{\F(A)})\subseteq \W_{\F(B)}$ and whenever $f:A \lrar B$ is a trivial fibration in $\M$ the associated right Quillen functor preserves weak equivalences, i.e., $f^*(\W_{\F(A)})\subseteq \W_{\F(B)}$.
\end{defn}

\begin{lem}(\cite[$3.8$]{HP})\label{characterization}
Let $\F:\M \lrar \ModCat$ be a proper relative functor.
\begin{enumerate}[(i)]
\item A morphism $(f,\phi):(A,X)\lrar (B,Y)$ is in $\Cof \cap \W$ if and only if $f:A\lrar B$ is a trivial cofibration and $\phi:f_{!}X\lrar Y$ is a trivial cofibration in $\F(B)$;
\item A morphism $(f,\phi):(A,X)\lrar (B,Y)$ is in $\Fib \cap \W$ if and only if $f:A\lrar B$ is a trivial fibration and $\phi^{ad}:X\lrar f^*Y$ is a trivial fibration in $\F(A)$.
\end{enumerate}
\end{lem}

We are now ready to state the main theorem of this subsection.
\begin{thm}\label{model structure}
Let $\M$ be a model category and $\F: \M \lrar \ModCat$ a proper relative functor. The classes of weak equivalences $\W$, fibrations $\Fib$ and cofibrations $\Cof$ of~\ref{d:model} endow $\int_{\M}\F$ with the structure of a model category, called the \textbf{integral model structure}.
\end{thm}

\subsection{Invariance of the integral model structure}\label{ss:invariance}

In this subsection we briefly recall the invariance property of the integral model structure, as appears in \cite[\S$2.2$]{HP}.
For the sake of clarity, we start with the categorical setting. Let $\AdjCat$ denote the $(2,1)$-category of categories and adjunctions. A morphism in $\AdjCat$ is an adjnuction (having the direction of the left adjoint) and a $2$-morphism is a pseudo-natural transformation of adjunctions which is an isomorphism in each component. 
We consider a diagram of categories
$$ \xymatrix{
\I \ar[dr]_\F\ar[rr]<1.3ex>^{\L} & & \J\ar[dl]^\G\ar[ll]<0.7ex>^{\R}_\upvdash \\ &  \AdjCat & \\}$$
such that the horizontal pair forms an adjunction. 

\begin{define}\label{d:left-right}
A \textbf{left morphism} from $\F$ to $\G$ over $\L \dashv \R$ is a pseudo-natural transformation $\F \Rightarrow \G \circ \L$, i.e., a compatible family of adjunctions.
$$\xymatrix{
\Sig^L_A:\F(A) \ar[r]<1ex> & \G(\L(A)):\Sig^R_A\ar[l]<1ex>_(0.55){\upvdash}.}$$
indexed by $A \in \I$. Similarly, a \textbf{right morphism} from $\F$ to $\G$ is a pseudo-natural transformation $\F \circ \R \Rightarrow \G$, i.e., a compatible family of adjunctions
$$\xymatrix{
\Theta^L_B:\F(\R(B)) \ar[r]<1ex> & \G(B):\Theta^R_B\ar[l]<1ex>_(0.45){\upvdash}.}$$
indexed by $B \in \J$.
\end{define}

\begin{rem}
Throughout we will be dealing with a pair of functors $\F,\G$ into $\AdjCat$ with different domains. To keep the notation simple, we shall, as before, use the notation $f_{!} \dashv f^*$ to indicate the image of a morphism $f$ under either $\F$ or $\G$. The possible ambiguity can always be resolved since $\F$ and $\G$ have different domains.
\end{rem}

We can now recall the model categorical counterpart.
\begin{defn}
Let $\M,\N$ be model categories and 
$$ \xymatrix{
\M \ar[dr]_\F\ar[rr]<1.3ex>^{\L} & & \N\ar[dl]^\G\ar[ll]<0.7ex>^{\R}_\upvdash \\ &  \ModCat & \\}$$
a diagram such that the horizontal pair is a Quillen adjunction and $\F,\G$ are proper relative functors. We will say that a left morphism $\F \Rightarrow \G \circ \L$ is a \textbf{left Quillen morphism} if the associated adjunctions
$$\xymatrix{
\Sig^L_A:\F(A) \ar[r]<1ex> & \G(\L(A)):\Sig_A^R\ar[l]<1ex>_(0.55){\upvdash}.}$$
are Quillen adjunction. Similarly we define \textbf{right Quillen morphisms}.
\end{defn}

\begin{define}\label{d:left-right-equiv}
Let $\M,\N,\F,\G$ be as above. We will say that a left Quillen morphism
$$\xymatrix{
\Sig^L_A:\F(A) \ar[r]<1ex> & \G(\L(A)):\Sig^R_A\ar[l]<1ex>_(0.55){\upvdash}.}$$
indexed by $A \in \M$ is a \textbf{left Quillen equivalence} if $\Sig^L_A \dashv \Sig^R_A$ is a Quillen equivalence for every cofibrant $A \in \M$. Similarly, we will say that a right Quillen morphism
$$\xymatrix{
\Theta^L_B:\F(\R(B)) \ar[r]<1ex> & \G(B) :\Theta^R_B\ar[l]<1ex>_(0.45){\upvdash}}$$
indexed by $B \in \N$ is a \textbf{right Quillen equivalence} if $\Theta^L_B \dashv \Theta^R_B$ is a Quillen equivalence for every fibrant $B \in \N$
\end{define}

Having established the necessary terminology, we can now state the invariance property:  
\begin{thm}\cite[$4.4$]{HP}\label{qa}
Let $\M,\N$ be model categories and
$$ \xymatrix{
\M \ar[dr]_\F\ar[rr]<1.2ex>^{\L} & & \N\ar[dl]^\G\ar[ll]<0.6ex>^{\R}_\upvdash \\ &  \ModCat & \\}$$
a diagram in which the horizontal pair is a Quillen adjunction and $\F,\G$ are proper relative functors. Let $\F \Rightarrow \G \circ \L$ be a left Quillen morphism given by a compatible family of adjunctions $\left(\Sig^L_A, \Sig^R_A\right)_{A \in \M}$. Then the induced adjunction
$$
\xymatrix{\Phi^L:\int_{\M}\F\ar[r]<1ex> & \int_{\N}\G:\Phi^R\ar[l]<1ex>_(0.5){\upvdash}.}
$$
is a Quillen adjunction. Furthermore, if the left Quillen morphism is a left Quillen equivalence then $\left(\Phi^L,\Phi^R\right)$ is a Quillen equivalence. The same result holds for the adjunction induced by a right Quillen morphism 
\end{thm}

\subsection{Integral model categories of group actions}\label{ss:global-strict}

Here we will recall two models for a global homotopy theory of group actions on spaces which were constructed in~\cite{HP}. In the first model we associated to each simplicial group $G$ the Borel model category $\SS^{\B G}$ and then used Theorem~\ref{model structure} to integrate the corresponding functor. In the second model we replaced simplicial groups with the equivalent model of \textbf{reduced simplicial sets}, under which a group corresponds to its classifying space. One can then assign to each reduced simplicial set $B$ the model category of spaces over $B$ and again use theorem~\ref{model structure} to integrate the associated functor. The equivalence between $G$-spaces and spaces over the classifying space of $G$, combined with Theorem~\ref{qa}, yields an Quillen equivalence between these two integral model categories.

Let $\sGr$ be the category of simplicial groups. This category admits a model structure which is transferred from the Kan-Quillen model structure on spaces via the adjunction
\begin{equation}
\xymatrix@=13pt{
\SS \ar[rr]<1ex>^(0.5){F} && \sGr \ar[ll]<1ex>_(0.5){\upvdash}^(0.5){U}
}
\end{equation}
where $U$ is the forgetful functor and $F$ is the free group functor. In particular, a map of simplicial groups $f: G \lrar H$ is a weak equivalence (resp. fibration) if and only if the map $U(f): U(G) \lrar U(H)$ is a weak equivalence (resp. fibration). 

For a simplicial group $G$ one can consider the category of spaces endowed with an action of $G$. This category can be identified with the simplicial functor category $\SS^{\B G}$ where $\B G$ is the simplicial groupoid with one object having $G$ as its automorphism group. As such, one can consider $\SS^{\B G}$ with the \textbf{projective model structure}, also called the \textbf{Borel model structure}. In this model structure a map of $G$-spaces is a weak equivalence (resp. fibration) if and only if it is such as a map of spaces. In addition, a $G$-space $X$ is cofibrant if and only if the action of $G$ on $X$ is \textbf{free} on each simplicial level (see~\cite{DDK} Proposition 2.2).

Let $f: G \lrar H$ be a map of simplicial groups. Then we have a Quillen adjunction
$$ \xymatrix@=13pt{
\SS^{\B G} \ar[rr]<1ex>^(0.5){f_{!}} && \SS^{\B H} \ar[ll]<1ex>_(0.5){\upvdash}^(0.5){f^*}
}$$
where $f_{!}(X) = H \times_G X$ is the quotient of $H \times X$ by the action of $G$ given by $g(h,x) = (hg^{-1},gx)$
and $f^*(X) = \res^H_G(X)$ is the restriction functor.

The following proposition appears, without proof, in~\cite{HP}.
\begin{pro}\label{p:cof}
If $f: G \lrar H$ is a cofibration of simplicial groups then $U(f): U(G) \lrar U(H)$ is a cofibration of simplicial sets. 
\end{pro}


\begin{proof}
We will show that the weakly saturated class generated by the maps $\{F(\Lam^n_i)\hrar F(\Del^n)\}$ is contained in the class of all monomorphisms of simplicial groups. Since any map $F(\Lam^n_i)\hrar F(\Del^n)$ is a monomorphism, it is enough to show that the class of monomorphisms is weakly saturated. Since this class is clearly closed under retracts and transfinite compositions, it remains to prove that if 
\begin{equation}\label{e:push-gp}
\xymatrix{K\ar@{^{(}->}[r]\ar[d] & L\ar[d]\\ X\ar[r] & P}
\end{equation} 
is a pushout diagram of simplicial groups in which the map $K\hrar L$ is a monomorphism, then the map $X\lrar P$ is a monomorphism as well.

Now the forgetful functor $U_{\Mon}:\sGp\lrar \sMon$ admits a right adjoint obtained by taking the sub-monoid of invertible elements in each simplicial degree. This implies that $U_{\Mon}$ preserves colimits and hence the square \ref{e:push-gp} is also a pushout square of simplicial monoids. Let $T_{\Mon}:\sGr\lrar \sMon$ be the composite of the forgetful functor $U:\sGr \lrar \sSet$ and the free simplicial monoid functor $\sSet \lrar \sMon$. We then have a pair of natural transformations
$$ U_{\Mon} \Rightarrow T_{\Mon} \Rightarrow U_{\Mon} $$
whose composite is the identity (i.e., the functor $U_{\Mon}$ is a retract of $T_{\Mon}$ in \\ $\Fun(\sGr,\sMon)$). Furthermore, the functor $T_{\Mon}$ sends monomorphisms to monomorohisms. We can then form the following pushout in $\sMon$

\begin{equation}\label{e:push-gp-2}
\xymatrix{T_{\Mon}(K)\ar@{^{(}->}[r]\ar[d] & T_{\Mon}(L)\ar[d]\\ U_{\Mon}(X)\ar[r] & P_T.}
\end{equation} 

Since $U_{\Mon}(K)\lrar U_{\Mon}(L)$ is a retract of $T_{\Mon}(K)\lrar T_{\Mon}(L)$ in the arrow category of $\sMon$ and we get that the map $U_{\Mon}(X)\lrar U_{\Mon}(P)$ is a retract of $U_{\Mon}(X)\lrar P_T$. By \cite[Theorem 4.1]{SS}, the map $U_{\Mon}(X)\lrar P_T$ is a monomorphism and so $U_{\Mon}(X)\lrar U_{\Mon}(P)$ is a monomorphism. This implies that $X \lrar P$ is a monomorphism as desired.

\end{proof} 

\begin{cor}(cf.~\cite[Proposition 6.4]{HP})\label{c:G-spaces}
The functor $\U: \sGr \lrar \ModCat$ given by $\U(G) = \SS^{\B G}$ is proper and relative so that we obtain an integral model structure on $\displaystyle \int_{G\in\sGp}\SS^{\B G}$.
\end{cor}
We shall refer to the model structure of Proposition \ref{c:G-spaces} as the \textbf{integral Borel model structure}.

Let us now recall the second model for global homotopy theory of group actions constructed in~\cite{HP}. As was established in~\cite{Kan}, the homotopy theory of simplicial groups may be equivalently described as the homotopy theory of \textbf{reduced simplicial sets}, i.e., simplicial sets $X \in \SS$ with $X_0 = \{*\}$. We will denote by $\SS_{0}$ the category of reduced spaces and by $\iota: \SS_0 \lrar \SS$ the full inclusion of reduced spaces in spaces. Then there exists a model strucutre on $\SS_0$  in which a map $f: X \lrar Y$ in $\SS_0$ is a weak equivalence (resp. cofibration) if and only if $\iota(f): \iota(X) \lrar \iota(Y)$ is a weak equivalence (resp. cofibration) in $\SS$ (\cite{Kan}, see also~\cite[\textrm{VI}.6.2]{GJ}). One then has a Quillen equivalence (see~\cite[$\mathrm{V}.6.3$]{GJ})
$$ \xymatrix@=13pt{\SS_0\ar[rr]<1ex>^(0.5){\GG} &&  \sGr \ar[ll]<1ex>_(0.5){\upvdash}^(0.5){\ovl{W}}} $$
where $\GG$ is the Kan loop group functor. Furthermore, for each simplicial group $G$ there exists a Quillen equivalence between $\SS^{\B G}$ and $\SS_{\iota(\ovl{W}(G))}$ (see~\cite{DDK}), where the latter is endowed with the corresponding slice model structure. In other words, the homotopy theory of $G$-spaces can be equivalently described as the homotopy theory of spaces over the classifying space $\ovl{W}(G)$.


Theorem~\ref{model structure} can then be applied to the functor \begin{equation}\label{e:V}\V:\SS_0 \lrar \ModCat ;\;\;B \mapsto \SS_{/\iota(B)},\end{equation} as is established in the following

\begin{pro}[\cite{HP}, Proposition $6.6$]\label{p:integral reduced}
The functor $\V$ above is proper and relative and gives rise to an integral model structure on $\displaystyle\int_{B\in \SS_0} \SS_{/\iota(B)}$.
\end{pro}

We now wish to compare the two models constructed above. In \cite[\S 6]{HP}, the authors showed that the Quillen equivalences of~\cite{DDK} assemble to form a right Quillen equivalence $\V \circ \ovl{W} \lrar \U$. In light of Theorem~\ref{qa} we then have
\begin{cor}\label{c:equivalence-1}
There exists a Quillen equivalence of integral model structures
$$
\xymatrix{\Phi^L: \displaystyle\mathop{\int}_{B \in \SS_0}\SS_{/\iota(B)} \ar[r]<2ex> & \displaystyle\mathop{\int}_{G \in \sGr}\SS^{\B G}:\Phi^R\ar[l]<0.7ex>_(0.5){\upvdash}.}
$$
\end{cor}

%

\section{An integral model structure for Segal group actions}\label{s:global-seg}

We shall now adapt the theory of Segal groups and Segal group actions to the setup of the model categorical Grothendieck construction. We will begin in \S\ref{ss:seg-group} by establishing a convenient model category for Segal groups. The construction we will use is based on a similar construction using topological spaces which appears in~\cite{Ber}. We will continue in \S\ref{ss:global-act} by showing that the functor which associates to each Segal group its corresponding model category of Segal group actions is proper and relative. For this purpose it will be convenient to identify the model structure constructed in~\cite{Pra} with a suitable slice model structure arising from work of Schwede, Shipley and Rezk. Finally, in \S\ref{ss:strict-vs} we will use the invariance theorem~\ref{qa} to show that the resulting integral model structure is equivalent to the two other model structures constructed in~\cite{HP}.

\subsection{A model structure for Segal groups}\label{ss:seg-group}

In this subsection we will construct a simplicial combinatorial model category whose fibrant-cofibrant objects are precisely the Segal groups. A similar construction (in the setting of topological spaces) was considered in~\cite{Ber}. 

Let $\sS = \SS^{\Del^{\op}}$ be the category of simplicial spaces. This category can be endowed with the injective model structure in which weak equivalences (resp. cofibrations) are the maps which are weak equivalences (resp. cofibrations) in each simplicial degree. This model structure coincides with the relevant \textbf{Reedy model structure} (see~\cite{Hir}).

\begin{notn}
Henceforth, we will abuse notation and write $\sS$ for the Reedy model structure on simplicial spaces. Other model structures on the category of simplicial spaces will be indicated by a subscript notation.
\end{notn} 

We shall denote by $|A_\bullet| = \diag(A_\bullet) \in \SS$ the realization of $A_\bullet$ in $\SS$. We will say that a simplicial space $A_\bullet$ is \textbf{reduced} if $A_0 = *$ is the one-pointed space. Let $\sS_0$ be the category of reduced simplicial spaces and $\iota: \sS_0 \lrar \sS$ the natural full inclusion. The functor $\iota$ admits a left adjoint
$$ (-)^{\red}: \sS \lrar \sS_0 $$
given by $A^{\red}_n = A_n/s^*(X_0)$ where $s$ is the unique map in $\Del$ from $[n]$ to $[0]$.

\begin{thm}[\cite{Ber}]
There exists a combinatorial model structure on $\sS_0$, which we shall call the \textbf{reduced Reedy model structure}, such that a map $f: A_\bullet \lrar B_\bullet$ is
\begin{enumerate}
\item
A weak equivalence if and only if $\iota(f):\iota(A_\bullet) \lrar \iota(B_\bullet)$ is a weak equivalence in each simplicial degree.
\item
A fibration if and only if $\iota(f):\iota(A_\bullet) \lrar \iota(B_\bullet)$ is a Reedy fibration.
\item
A cofibration if and only if $\iota(f):\iota(A_\bullet) \lrar \iota(B_\bullet)$ is a Reedy cofibration, i.e., if and only if it is a monomorphism.
\end{enumerate}
\end{thm}
\begin{proof}
We first claim that the composed functor $\iota((-)^{\red}): \sS \lrar \sS$ preserves cofibrations and trivial cofibrations. Let $f: A_\bullet \lrar B_\bullet$ be a (trivial) cofibration. By standard properties of the Reedy model structure on simplicial spaces the induced map
$$ 
\xymatrix{
& B_0 \ar[dr]\ar[dl] & \\
A_n \coprod_{A_0} B_0 \ar[rr] && B_n\\
} $$
is a (trivial) cofibration of spaces under $B_0$. Since the pushforward functor
$$ \SS_{B_0/} \lrar \SS_{*/} $$
given by $K \mapsto K/B_0$ is a left Quillen functor we get that the induced map
$$ A_n/A_0 \cong \left(A_n \coprod_{A_0} B_0\right)/B_0 \lrar B_n/B_0 $$
is a (trivial) cofibration for every $n$.

Combining this result with the fact that the inclusion $\iota: \sS_0 \lrar \sS$ preserves pushouts and filtered colimits one can use the transfer lemma (see~\cite[Proposition $1.12.1$]{Ber}) to transfer the Reedy model structure from $\sS$ to $\sS_0$. The resulting model structure will satisfy $(1)$ and $(2)$ above by definition. Furthermore, we will get in addition that the functor $\iota$ preserves cofibrations. To prove that $\iota$ also reflects cofibrations it is enough to note that the unit map
$$ A_\bullet \lrar (\iota(A_\bullet))^{\red} $$
is an isomorphism for every $A \in \sS_0$.
\end{proof}

\begin{notn}
Henceforth, we will abuse notation and write $\sS_0$ for the reduced Reedy model structure on reduced simplicial spaces. Other model structures on the category of reduced simplicial spaces will be indicated by a subscript notation.
\end{notn} 

\begin{define}
Given a reduced simplicial space $A_\bullet$ and a space $K \in \SS$ we shall define
$$ K \otimes A_\bullet = (K \times A_\bullet)^{\red} $$
where $K \times A_\bullet \in \sS$ is the simplicial space defined by $(K \times A)_n = K \times A_n$.
\end{define}

\begin{lem}
The association $(K,A_\bullet) \mapsto K \otimes A_\bullet$ determines a simplicial model structure on the reduced Reedy model category $\sS_0$.
\end{lem}
\begin{proof}
We need to check that the pushout product axiom is satisfied. Recall that the Reedy model structure on $\sS$ is simplicial (with respect to the level-wise product). The result now follows directly from the fact that the reducification functor $(-)^{\red}: \sS \lrar \sS_0$ is simplicial by definition and preserves Reedy cofibrations, trivial Reedy cofibrations and pushouts.
\end{proof}

\begin{define}\label{d:f-i-n}
For each $n \geq 1$ and $0 \leq i \leq n$ we will denote by
$$ f_{n,i}: \Del^{\{0,...,n\}} \coprod_{\Del^{i}} \Del^{\{i,n+1\}} \hrar \Del^{n+1} $$
the corresponding inclusion of (level-wise discrete) simplicial spaces. We will denote by $f^{\red}_{n,i}$ the map of reduced simplicial spaces obtained by applying the reducification functor $(-)^{\red}: \sS \lrar \sS_0$. We will denote by 
$$ \mathbb{S} = \left\{f^{\red}_{n,i}\right\} $$ 
this set of maps in $\sS_0$. 
\end{define}

\begin{define}
We will say that a reduced simplicial space $A_\bullet$ is a \textbf{Segal group} if it is a Segal space in the sense of~\cite{Rez} in which every morphism is invertible (see~\cite[\S 5.5]{Rez}). Note that in this definition a Segal group is always fibrant in the reduced Reedy model structure.
\end{define}

\begin{rem}
If $A_\bullet$ is a reduced Segal space then $\pi_0(A_1)$ inherits a natural structure of a monoid. The condition that every morphism is in $A_\bullet$ is invertible is equivalent to the monoid $\pi_0(A_1)$ being a group.
\end{rem}

\begin{pro}\label{p:segal}
Let $A_\bullet$ be a reduced Reedy fibrant simplicial space. Then $A_\bullet$ is a Segal group if and only if $A_\bullet$ is \textbf{local} with respect to $\mathbb{S}$, i.e., if for every 
map $f: S_\bullet \lrar T_\bullet$ in $\mathbb{S}$ the induced map on simplicial mapping spaces
$$ \Map_{\sS_0}\left(T_\bullet,A_\bullet\right) \lrar \Map_{\sS_0}\left(S_\bullet,A_\bullet\right) $$
is a weak equivalence of spaces.
\end{pro}
\begin{proof}
Given $0 \leq i < j \leq n$ we will denote by $\Sp_{i,j} \subseteq \Del^n$ the simplicial subset given by
$$ \Sp_{i,j} = \Del^{\{i,i+1\}} \coprod_{\Del^{\{i+1\}}} \Del^{\{i+1,i+2\}} \coprod_{\Del^{\{i+2\}}} ... \coprod_{\Del^{\{j-1\}}} \Del^{\{j-1,j\}} \subseteq \Del^n .$$
Since $A_\bullet$ is Reedy fibrant we note that $A_\bullet$ is a Segal space if and only if it is local with respect to the inclusion $\Sp_{0,n} \subseteq \Del^n$ for every $n \geq 2$. 

Now assume that $A_\bullet$ is local with respect to $\mathbb{S}$. Consider the sequence of inclusions
$$ \Sp_{n+1} \subseteq \Del^{\{0,1,2\}} \coprod_{\Del^{\{2\}}} \Sp_{2,n+1} \subseteq \Del^{\{0,1,2,3\}} \coprod_{\Del^{\{3\}}} \Sp_{3,n+1} \subseteq ... \subseteq \Del^{\{0,...,n\}} \coprod_{\Del^{\{n\}}} \Del^{\{n,n+1\}} \subseteq \Del^{\{n+1\}} .$$
Then each map in this sequence is a pushout along $f_{j,j} \in \mathbb{S}$ for some $1 \leq j \leq n$. Since $A_\bullet$ is $\mathbb{S}$-local we deduce that $A_\bullet$ is local with respect to $\Sp_{n+1} \subseteq \Del^{n+1}$ for every $n \geq 1$ and so $A_\bullet$ is a Segal space. Let us now show that the monoid $\pi_0(A_1)$ is a group. Since $A_\bullet$ is local with respect to $f_{1,0},f_{1,1} \in \mathbb{S}$ we obtain a span of weak equivalences 
$$ \xymatrix{
& A_2 \ar_{\simeq}^{f_{1,0}^*}[dr]\ar^{\simeq}_{f_{1,1}^*}[dl] & \\
A_1 \times A_1 && A_1 \times A_1 \\
}$$
This implies that the shearing map
$$ \pi_0(A_1) \times \pi_0(A_1) \lrar \pi_0(A_1) \times \pi_0(A_1) $$
given by $(a,b) \mapsto (a,ab)$ is a bijection, so that $\pi_0(A_1)$ is a group.

Now assume that $A_\bullet$ is a Segal group. Then $A_\bullet$ is a Segal space and is hence local with respect to $\Sp_{n+1} \subseteq \Del^{n+1}$ for every $n \geq 1$. In particular, $A_\bullet$ is local with respect to $f_{1,1} \in \mathbb{S}$. Now let $T_{n,i} \subseteq \Del^{n+1}$ be the simplicial subset given by the edges $\Del^{\{j,j+1\}}$ for $j=0,...,i-1$ and the triangles $\Del^{\{j,j+1,n+1\}}$ for $j=i,...,n-1$. Then $T_{n,i}$ contains $\Sp_{n+1}$ and can be obtained from $\Sp_{n+1}$ by a sequence of pushouts along $f_{1,1}$. This means that $A_\bullet$ is local with respect to the inclusion $\Sp_{n+1} \subseteq T_{n,i}$ and so $A_\bullet$ is local with respect to the inclusion $T_{n,i} \subseteq \Del^{n+1}$ as well. Now consider the diagram
$$ \xymatrix{
\Sp_{0,n} \coprod_{\Del^{\{i\}}}\limits \Del^{\{i,n+1\}} \ar[r]\ar[d] & T_{n,i} \ar[d] \\
\Del^{\{0,...,n\}} \coprod_{\Del^{\{i\}}}\limits \Del^{\{i,n+1\}} \ar^-{f_{n,i}}[r] & \Del^{n+1} \\
}$$
From the above considerations we see that $A_\bullet$ is local with respect to both vertical maps. Hence to show that $A_\bullet$ is local with respect $f_{n,i}$ is equivalent to showing that $A_\bullet$ is local with respect to the upper horizontal map. Now $T_n$ can be obtained from $\Sp_{0,n} \coprod_{\Del^{\{i\}}}\limits \Del^{\{i,n+1\}}$ by performing pushouts along $f_{1,0}$, and so it will suffice to show that $A_\bullet$ is local with respect to $f_{1,0}$. This, in turn, follows directly from the characterization of invertible morphisms in a Segal space appearing in~\cite[\S 5.5]{Rez}.

\end{proof}


Since the reduced Reedy model structure on $\sS_0$ is left proper and combinatorial we can take its left Bousfield localization with respect to $\mathbb{S}$. We shall call the localized model structure the \textbf{Segal group model structure} and denote it by $\Seg$. In light of Proposition~\ref{p:segal}, this (simplicial, combinatorial) model structure satisfies the following properties:
\begin{enumerate}
\item
Weak equivalences in the Segal group model structure are the maps $f: A_\bullet \lrar B_\bullet$ such that for every Segal group $S_\bullet$ the induced map
$$ \Map_{\sS_0}\left(B_\bullet,S_\bullet\right) \lrar \Map_{\sS_0}\left(A_\bullet,S_\bullet\right) $$
is a weak equivalence.
\item
The cofibrations are the monomorphisms.
\item
An object $A_\bullet$ is fibrant if and only if it is a Segal group.
\end{enumerate}

\begin{rem}\label{r:locality}
In light of Proposition~\ref{p:segal}, it follows from the general theory of left Bousfield localization that if $f: A_\bullet \lrar B_\bullet$ is a weak equivalence in the Segal group model structure and both $A_\bullet,B_\bullet$ are Segal groups then $f$ is a weak equivalence in each simplicial degree.
\end{rem}

Note that if $A_\bullet$ is a reduced simplicial space then $|A_\bullet|$ is a reduced simplicial set. We shall hence consider the realization functor also as a functor $|-|:\sS_0 \lrar \SS_0$. As such, it admits a right adjoint
$$ \Om: \SS_0 \lrar \sS_0 $$
where $\Om(A)_n$ is given by the simplicial mapping space of pairs 
$$ \Om(A)_n = \Map((\Del^n,(\Del^n)_0),(A,A_0)) .$$
Furthermore, the adjunction
\begin{equation}\label{e:adj-reedy} 
\xymatrix@=13pt{\SS_0\ar[rr]<1ex>^(0.6){|-|} &&  \SS_0 \ar[ll]<1ex>_(0.4){\upvdash}^(0.4){\Om}}
\end{equation}
is a simplicial Quillen adjunction (where $\SS_0$ is endowed with the reduced Reedy model structure).


\begin{rem}\label{r:easy-side}
Note that for any $f_{n,i}$ as in Definition~\ref{d:f-i-n} the induced map $|f_{n,i}|$ is a weak equivalence of (non-reduced) spaces, and induces an isomorphism on the sets of vertices. It then follows that $|f^{\red}_{n,i}|$ is a weak equivalence of reduced spaces. By adjunction we get that $\Om(S)_\bullet$ is a Segal group for any fibrant reduced simplicial set $S$. Using adjunction again it follows that the realization functor sends every weak equivalence in $\Seg$ to a weak equivalence in $\SS_0$.
\end{rem}


\begin{cor}\label{c:qe-reduced}
The adjunction~\ref{e:adj-reedy} descends to a simplicial Quillen adjunction
$$ \xymatrix@=13pt{\Seg\ar[rr]<1ex>^(0.6){|-|} &&  \SS_0 \ar[ll]<1ex>_(0.4){\upvdash}^(0.4){\Om}} $$
\end{cor}
\begin{proof}
Clearly $|-|$ preserves cofibrations. From Remark~\ref{r:easy-side} we see that $|-|$ preserves trivial cofibrations as well.
\end{proof}

The following converse of Remark~\ref{r:easy-side} is essentially contained in~\cite{Seg}.
\begin{thm}[\cite{Seg}]\label{t:segal}
If $A_\bullet$ is a Segal group then the natural map
$$ A_\bullet \lrar \Om\left(|A_\bullet|^{\fib}\right)_\bullet $$
is a level-wise equivalence.
\end{thm}

\begin{cor}\label{c:realization}
A map $f: A_\bullet \lrar B_\bullet$ is a weak equivalence in the Segal group model structure if and only if it induces a weak equivalence
$$ |f|:|A_\bullet| \lrar |B_\bullet| $$
in $\SS_0$.
\end{cor}
\begin{proof}
Since the Quillen adjunction in question is simplicial we see that $|f|$ is a weak equivalence in $\SS_0$ if and only if $f$ induces an equivalence
$$ \Map_{\sS_0}\left(B_\bullet,\Om(S)_\bullet\right) \lrar \Map_{\sS_0}\left(A_\bullet,\Om(S)_\bullet\right) $$
for every fibrant reduced space $S \in \SS_0$. The result now follows from Theorem~\ref{t:segal}.
\end{proof}

\begin{cor}\label{c:qe-reduced-equiv}
The simplicial Quillen adjunction
$$ \xymatrix@=13pt{\Seg\ar[rr]<1ex>^(0.6){|-|} &&  \SS_0 \ar[ll]<1ex>_(0.4){\upvdash}^(0.4){\Om}} $$
is a Quillen equivalence.
\end{cor}
\begin{proof}
From Corollary~\ref{c:realization} we get that $|-|$ preserves and reflects weak equivalences. It is hence enough to check that for every fibrant object $A \in \SS_0$ the unit map
$$ u:A \lrar |\Om(A)_\bullet| $$
is a weak equivalence. Since $\Om(A)_\bullet$ is a Segal group we get from Theorem~\ref{t:segal} that the map $u$ induces an equivalence on loop spaces. Since $A$ and $|\Om(A)_\bullet|$ are reduced (and in particular connected) this implies that $u$ is a weak equivalence.
\end{proof}

\subsubsection{ Relation with Bousfield's approach}

We have seen that Segal groups constitute a model for pointed connected spaces and hence for simplicial groups. There is another approach, based on an unpublished manuscript of A.~K.~Bousfield, which we now describe.

\begin{define}\label{d:bousfield}
Let $B_n \subseteq \Del^n$ be the simplicial subset given by the edges $\Del^{\{0,i\}}$ for $i=1,...,n$. We shall refer to $B_n$ as the $n$'th \textbf{Bousfield spine}.
\end{define}

The terminology of the following definition is taken from~\cite{Bergn}.

\begin{define}[\cite{Bou}]
A \textbf{reduced Bousfield-Segal space} is a Reedy fibrant reduced simplicial space $A_\bullet$ which is local with respect to the inclusion $B_n \subseteq \Del^n$ for every $n$. In particular, for every $n$ the induced map
$$ A_n = \Map(\Del^n,A) \lrar \Map(B_n,A) = A_1^n $$
is a weak equivalence.
\end{define}

In~\cite[Theorem $3.1$]{Bou} Bousfield proves that the adjunction
\begin{equation}
\xymatrix@=13pt{\sS_0\ar[rr]<1ex>^(0.6){|-|} &&  \SS_0 \ar[ll]<1ex>_(0.4){\upvdash}^(0.4){\Om}}
\end{equation}
induces an equivalence between the homotopy category of Bousfield-Segal spaces and the homotopy category of Kan simplicial sets. In particular, he shows that a reduced simplicial space is a Bousfield-Segal space if and only if it is levelwise equivalent to a reduced simplicial space of the form $\Om(X)_\bullet$ for some Kan simplicial set $X$. It then follows from Corollary~\ref{c:qe-reduced-equiv} and Theorem~\ref{t:segal} that a reduced simplicial space is a Bousfield-Segal space if and only if it is a Segal group.

Since the reduced Reedy model structure on $\sS_0$ is left proper and combinatorial one can take its left Bousfield localization with respect to the inclusions $B_n \hrar \Del^n$. The resulting model category was considered in~\cite{Bergn}. However, in light of~\cite{BergnE}, the argument appearing in~\cite[\S 6]{Bergn}, comparing this model category to the model category of simplicial groups, is only partial. On the other hand, it follows directly from the arguments above that this model category is in fact equal to the model category $\Seg$, and hence equivalent to the model category of simplicial groups by Corollary~\ref{c:qe-reduced-equiv}.

\subsection{Construction of the integral model structure }\label{ss:global-act}

Our goal in this subsection is to show that the model categories of Segal group actions constructed in~\cite{Pra} can be assembled to form an integral model structure. Let us begin with the basic definitions.

\begin{define}{\cite[Definition $3.1$]{Pra} (cf. \cite[3.1,3.4]{NSS})}\label{SGA}
Let $A_\bullet$ be a Segal group. A \emph{Segal group action} is a Reedy fibration of simplicial spaces $$p: X_\bullet\lrar \iota (A_\bullet) $$
such that for every $n$, the maps $$\xymatrix{X_n\ar@{>}[rr]<0.5ex>^(0.25){\sigma_{\{n\}}^*\times p_n} \ar@{>}[rr]<-0.5ex>_(0.25){\sigma_{\{0\}}^*\times p_n}
 & & X_0\times_{A_0} A_n\simeq X_0\times A_n,}$$ induced by the maps $\sigma_{\{0\}},\sigma_{\{n\}}:[0]\lrar [n]$, $0\mapsto 0$ and $0\mapsto n$,\\ 
are weak equivalences. Note that the $0$'th space $X_0$ should be considered as the underlying space on which the the loop space $A_1\simeq \Om |A_\bullet|$ coherently acts.

\end{define}

Let $\Del_{\inj} \subseteq \Del$ denote the subcategory consisting of all objects and all injective maps.
\begin{define}
A map of simplicial spaces $f: A_\bullet \lrar B_\bullet$ is said to be \textbf{equifibred} if for every map $[m] \lrar [n]$ in $\Del_{\inj}$ the induced square
$$ \xymatrix{
A_n \ar[r]\ar[d] & B_n \ar[d] \\
A_m \ar[r] & B_m \\
}$$
is homotopy Cartesian.
\end{define}

\begin{rem}
Since any map $\rho: [n] \lrar [m]$ in $\Del_{\inj}$ can be extended to a retract diagram $[n] \lrar [m] \lrar [n]$ inside $\Del$, it follows from the pasting lemma for homotopy Cartesian squares that a map of simplicial spaces $f: A_\bullet \lrar B_\bullet$ is equifibred if and only if for every map $[m] \lrar [n]$ in $\Del$ the induced square
$$ \xymatrix{
A_n \ar[r]\ar[d] & B_n \ar[d] \\
A_m \ar[r] & B_m \\
}$$
is homotopy Cartesian.
\end{rem}

\begin{pro}\label{p:char}
Let $A_\bullet$ be a Segal group. Then a map $X_\bullet \lrar \iota(A_\bullet)$ of simplicial spaces is a Segal group action if and only if it is an equifibred Reedy fibration.
\end{pro}
\begin{proof}

Clearly every equifiberd Reedy fibration is a Segal group action. To prove the converse, let $f:X_\bullet \lrar \iota(A_\bullet)$ be a Segal group action. We will say that a map $\sig: [m] \lrar [n]$ in $\Del$ is $f$-Cartesian if the diagram
$$ \xymatrix{
X_n \ar[r]\ar[d] & A_n \ar[d] \\
X_m \ar[r] & A_m \\
}$$
is homotopy Cartesian. In order to prove the desired result we need to show that every map in $\Del_{\inj}$ is $f$-Cartesian. Since $f$ is a Segal group action we have by Definition~\ref{SGA} that $\sig_{\{0\}}: [0] \lrar [n]$ and $\sig_{\{n\}}: [0] \lrar [n]$ are $f$-Cartesian. From the pasting lemma for homotopy Cartesian squares we get that the set of $f$-Cartesian maps is closed under composition. Moreover, we know that if $\sig$ and $\tau \circ \sig$ are $f$-Cartesian then $\tau$ is $f$-Cartesian as well. Using these arguments we can prove the desired result in three steps:
\begin{enumerate}
\item
Since $\sig_{\{0\}}: [0] \lrar [m]$ and $\sig_{\{0\}}: [0] \lrar [n]$ are $f$-Cartesian we get that $\sig_{\{0,...,m\}}: [m] \lrar [n]$ is $f$-Cartesian.
\item
Since $\sig_{\{m\}}: [0] \lrar [m]$ and $\sig_{\{0,...,m\}}: [m] \lrar [n]$ are $f$-Cartesian we get that $\sig_{\{m\}}: [0] \lrar [n]$ is $f$-Cartesian.
\item
Since $\sig_{\{0\}}: [0] \lrar [m]$ and $\sig_{\{i\}}: [0] \lrar [n]$ are $f$-Cartesian it follows that every injective map $\sig: [m] \lrar [n]$ such that $\sig(0) = i$ is $f$-Cartesian.
\end{enumerate}
\end{proof}

We now wish to construct a model for coherent group actions using $\Seg$. For this, we shall consider the following model structure introduced by Rezk, Schwede and Shipley (\cite{RSS}):
\begin{thm}[\cite{RSS}]\label{t:rezk}\
\begin{enumerate}
\item
There exists a model structure on $\sS$, denoted $\sS_{\equ}$, such that
\begin{enumerate}
\item
The weak equivalences are the maps $f: X_\bullet \lrar Y_\bullet$ such that the induced map $|f|: |X_\bullet| \lrar |Y_\bullet|$ is a weak equivalence in $\SS$.
\item
The cofibrations are the monomorphisms.
\item
The fibrations are the equifibred Reedy fibrations.
\end{enumerate}
\item
The realization functor $|-|: \sS \lrar \SS$ fits into a Quillen equivalence
$$ \xymatrix@=13pt{\sS\ar[rr]<1ex>^(0.5){|-|} &&  \SS \ar[ll]<1ex>_(0.5){\upvdash}^(0.5){C}} $$
where $C(A)_n = \Map(\Del^n,A)$.
\end{enumerate}
\end{thm}
\begin{proof}
The first part is a particular case of Theorem $3.6$ of~\cite{RSS}. As for part $(2)$, it is clear that $|-|$ preserves cofibrations and trivial cofibrations. Furthermore, $|-|$ preserves and detects weak equivalences. Since every object is cofibrant it will be enough to prove that for every fibrant $B \in \SS$ the counit map
$$ |C(B)_\bullet| \lrar B $$
is a weak equivalence. But this follows from the fact that if $B$ is fibrant, $C(B)_\bullet$ is homotopy constant with value $B$.
\end{proof}

Now given an object $A_\bullet \in \sS_0$ we will consider the model category $\sS_{/\iota(A_\bullet)}$ endowed with the slice model structure inherited from $\sS_{\equ}$ and denote it by $\l(\sS_{/\iota(A_\bullet)}\r)_{\act}$. 
This notation is justified by the following Corollary of Proposition~\ref{p:char}.
\begin{cor}\label{c:char-2}
Let $A_\bullet$ be a Segal group. Then the fibrant(-cofibrant) objects of $\l(\sS_{/\iota(A_\bullet)}\r)_{\act}$ are precisely the Segal group actions.
\end{cor}

\begin{rem}\label{r:jardine}
Note that every fibration in $\sS_{\equ}$ 
is in particular a Kan fibration in the sense of Jardine (see~\cite[Lemma $2.4$]{Jar}). Hence according to~\cite[Theorem $2.14$]{Jar} we see that the realization functor $|-|: \sS \lrar \SS$ preserves fibrations. Since $\SS$ is right proper one can deduce that $\sS$ is right proper as well. 
\end{rem}

Combining Remark~\ref{r:jardine} with \cite[Corollary 6.2]{HP} we hence obtain the following
\begin{cor}
The slice functor $\sS_{/(-)}:\sS \lrar \ModCat$ is proper and relative. 
\end{cor}

We warn the reader that the functor $\iota: \sS_0 \lrar \sS_{\equ}$ is \textbf{not} a Quillen functor. However, it does preserve weak equivalences, trivial cofibrations and trivial fibrations (the latter consists of trivial Reedy fibrations in both cases). In light of this and the above remark, it follows that the association $A_\bullet \mapsto \sS_{/\iota(A_\bullet)}$ does determine a proper relative functor \begin{equation}\label{e:W}\W:\sS_0 \lrar \ModCat. \end{equation} We hence obtain the following conclusion:
\begin{cor}
The integral model structure $$\displaystyle\mathop{\int}_{A_\bullet\in\l(\sS_0\r)_{\segal}}\left(\sS_{/\iota(A_\bullet)}\right)_{\act}$$ exist and its fibrant-cofibrant objects are precisely the Segal group actions. 
\end{cor}

\begin{rem}
Note that a map of Segal group actions over different Segal groups is simply a commutative square of simplicial spaces
$$\xymatrix{X_\bullet \ar[r]\ar[d] &A_\bullet\ar[d]\\ Y_\bullet\ar[r] & B_\bullet.}$$ 
As we shall see later, under the Quillen equivalence of Corollary \ref{c:equivalence-3} such a map corresponds to a map of (ordinary) group actions $(G,X)\lrar (H,Y)$, namely a map of simplicial groups $G\lrar H$, and a map of spaces $X\lrar Y$ which is $G$-equivariant when $Y$ is considered as a $G$-space via the map $G\lrar H$.
\end{rem}%

\subsection{Rectification of Segal group actions}\label{ss:strict-vs}

The purpose of this subsection is to prove that the integral model category constructed in \S\ref{ss:global-act} is equivalent to the two model structures recalled in \S\ref{ss:global-strict}. This can be viewed as a rectification theorem for Segal group actions in a global setting.

In light of Theorem~\ref{qa}, we need to provide a compatible family of Quillen adjunctions

$$\xymatrix{
\Xi^L_A:\l(\sS_{/\iota(\Om(A)_\bullet})\r)_{\act} \ar[r]<1ex> & \SS_{/\iota (A)}:\Xi^R_A\ar[l]<1ex>_(0.4){\upvdash}}$$
indexed by $A \in \SS_0$, which are equivalences whenever $A$ is fibrant. This can be done as follows. Let $A \in \SS_0$ be a reduced space and $X$ a space equipped with a map $X \lrar \iota (A)$. Define $\Xi^R_A(X)$ to be the pullback in the square
$$ \xymatrix{
\Xi^R_A(X) \ar[r]\ar[d] &  \ar[d] C(X)_\bullet \\
\iota(\Om(A))_\bullet \ar[r] & C(\iota (A))_\bullet \\
}$$
where $C(-)_\bullet = \Map(\Del^\bullet,-): \SS \lrar \sS$ is the functor discussed in Theorem~\ref{t:rezk} $(2)$. Dually, if $f:X_\bullet \lrar \iota(\Om(A)_\bullet)$ is an object of $\sS_{/\iota(\Om(A)_\bullet)}$ then $\Xi^L_A(X_\bullet)$ is given by 
$$ \Xi^L_A(X_\bullet) = |X_\bullet| $$ 
with the map $|X_\bullet| \lrar A$ given by the composition
$$ |X_\bullet| \lrar \left|\left(X^{\red}_\bullet\right)\right| \lrar A $$
where the second map is the adjoint of the map $X^{\red}_\bullet \lrar \Om(A)_\bullet$ which in turn is the adjoint of $f$.

It is straightforward to verify that $\Xi^L_A \dashv \Xi^R_A$ forms an adjunction. Furthermore, since the right Quillen functors for both $\W$ and $\V$ are defined via pullbacks (see \ref{e:W} and \ref{e:V}), it is straightforward to verify that $\Xi^L_A \dashv \Xi^R_A$ carries natural compatibility isomorphisms exhibiting it as a right Quillen morphism. We now claim that this is actually a right Quillen equivalence.

\begin{pro}\label{p:qe-reduced-2}
For every fibrant $A \in \SS_0$, the adjunction
$$\xymatrix{
\Xi^L_A:\l(\sS_{/\iota(\Om(A)_\bullet})\r)_{\act} \ar[r]<1ex> & \SS_{/\iota A}:\Xi^R_A\ar[l]<1ex>_(0.4){\upvdash}}$$
is a Quillen equivalence.
\end{pro}
\begin{proof}
It is evident from the definitions that $\Xi^L_A$ preserves cofibrations and trivial cofibrations. Furthermore, the functor $\Xi^L_A$ preserves and reflects weak equivalences. Since every object in $\sS_{/\iota\Om(A)}$ is cofibrant it will now suffice to show that for every fibrant object $X \lrar \iota A$ in $\SS_{/\iota A}$ the counit map
$$ \Xi_A^L\left(\Xi^R_A(X)\right) \lrar X $$
is a weak equivalence. Consider the diagram
$$ \xymatrix{
\left|\Xi^R_A(X)\right| \ar[r]\ar[d] &  \ar[d] \left|C(X)_\bullet\right| \ar^(0.6){\sim}[r] & X \ar[d] \\
\left|\iota\Om(A)_\bullet\right| \ar[r] & \left|C(A)_\bullet\right| \ar^(0.6){\sim}[r] & A\\
}$$
Since $X$ is fibrant in $\SS_{/\iota A}$ we know that the map $X \lrar\iota(A)$ is a fibration in $\SS$. Since $A$ in turn is fibrant in $\SS_0$ it is also fibrant in $\SS$ (\cite[Lemma $\mathrm{V}.6.6$]{GJ}) and so $X$ is fibrant in $\SS$. From Theorem~\ref{t:rezk} $(2)$ we then deduce that the maps on the right square are weak equivalences. 

Since $A$ is fibrant we get from Corollary~\ref{c:qe-reduced-equiv} that the composition of the bottom horizontal maps is a weak equivalence and hence the left-bottom horizontal map is a weak equivalence by $2$-out-of-$3$. 

Now since the realization functor $|-|$ preserves pullbacks the left square is a pullback square. Since $C(X)_\bullet \lrar C(A)_\bullet$ is a fibration in $\sS$ it is also a Kan fibration in the sense of Jardine~\cite{Jar}. Using the same argument as in Remark~\ref{r:jardine} we obtain that the middle vertical map is a fibration. Since $\SS$ is right proper we deduce that the map
$$ \Xi_A^L\left(\Xi^R_A(X)\right) = \left|\Xi^R_A(X)\right| \lrar \left|C(X)_\bullet\right| $$
is a weak equivalence and hence the composite
$$ \Xi_A^L\left(\Xi^R_A(X)\right) \lrar X $$
is a weak equivalence as well.
\end{proof}

In light of Corollary~\ref{c:qe-reduced-equiv}, Proposition~\ref{p:qe-reduced-2} and Theorem~\ref{qa} we obtain the following
\begin{cor}\label{c:equivalence-2}
There exists a Quillen equivalence
$$
\xymatrix{\Psi^L: \displaystyle\mathop{\int}_{A_\bullet \in \l(\sS_0\r)_{\segal}}\l(\sS_{/\iota(A_\bullet )}\r)_{\act} \ar[r]<1ex> & \displaystyle\mathop{\int}_{B \in \SS_0}\SS_{/\iota B}:\Psi^R\ar[l]<1ex>_(0.4){\upvdash}.}
$$
\end{cor}
For future reference we record the following
\begin{cor}\label{c:equivalence-3}
The Quillen equivalences of Corollaries \ref{c:equivalence-1} and \ref{c:equivalence-2} compose and yield a Quillen equivalence 
$$
\xymatrix{\Lambda^L: \displaystyle\mathop{\int}_{A_\bullet \in \l(\sS_0\r)_{\segal}}\l(\sS_{/\iota(X_\bullet )}\r)_{\act} \ar[r]<1ex> & \displaystyle\mathop{\int}_{G \in \sGr}\SS^{\B G}:\Lambda^R\ar[l]<1ex>_(0.4){\upvdash}.}
$$

between the integral model structures for strict group actions and Segal group actions.

\end{cor}

\section{Truncation theory for integral model categories}

Recall that a space $X \in \SS$ is called \textbf{$n$-truncated} if it has no homotopy groups above dimension $n$. It is a classical fact that $X$ is $n$-truncated if and only if $\Map^{\der}_\SS(Y,X)$ is $n$-truncated for every space $Y \in \SS$. The concept of $n$-truncation can hence be generalized as follows (see also~\cite[\S 5.5.6]{Lur09} for the notion of $n$-truncation in the $\infty$-categorical setting).
\begin{define}\label{d:truncated}
Let $\M$ be a model category. We will say that an object $X \in \M$ is \textbf{$n$-truncated} if for every object $Y \in \M$ the derived mapping space $\Map^{\der}_\M(Y,X)$ is an $n$-truncated space. We will say that a map $f: X \lrar X_n$ exhibits $X_n$ as an \textbf{$n$-truncation} of $X$ if for every $n$-truncated object $Z$ the induced map
$$ \Map^{\der}_\M(X_n,Z) \lrar \Map^{\der}_\M(X,Z) $$
is a weak equivalence.
\end{define}

\begin{example}\label{e:slice}
Let $A \in \SS$ be a space and consider the slice model category $\SS_{/A}$. Then an object $f:X \lrar A$ in $\SS_{/A}$ is $n$-truncated if and only if the homotopy fiber of $f$ is $n$-truncated over every point $a \in A$.
\end{example}

In this section we will give a description of $n$-truncated objects and $n$-truncation maps in a general integral model structure $\int_\M \F$. We start with the following proposition which describes the behaviour of derived mapping spaces in the integral model structure. It is well-known that derived mapping spaces in a model category depend only on the underlying relative category. Since the underlying $\infty$-category of the integral model structure coincides with the $\infty$-categorical Grothendieck construction of the underlying $\infty$-functor, it will be convenient to prove the following proposition using the machinery of $\infty$-categories. 

Let $\Set^+_\Del$ denote the category of marked simplicial sets. The category $\Set^+_\Del$ can be endowed with the coCartesian model structure (see~\cite[Remark $3.1.3.9$]{Lur09}) yielding a model for the theory of $\infty$-categories. Given a marked simplicial set $(\C,\V)$ we will denote by $\L(\C,\V)$ the fibrant replacement of $(\C,\V)$ in $\Set^+_\Del$. We will consider $\L(\C,\V)$ as a model for the $\infty$-localization of $\C$ obtained by formally inverting the arrows of $\V$. 

Now let $\M$ be a model category. We will denote
$$ \M_\infty \x{\df}{=} \L\left(\N\left(\M^{\cof}\right),\N\left(\W\cap \M^{\cof}\right)\right) .$$ 
Here, $\M^{\cof} \subseteq \M$ denotes the full subcategory of cofibrant objects. Following Lurie (see~\cite[Definition $1.3.4.15$]{Lur14}), we will refer to $\M_\infty$ as the \textbf{underlying $\infty$-category of $\M$}. In the case of $\M = \Set^+_\Del$, we will also denote by $\Cat_\infty \x{\df}{=} \left(\Set^+_\Del\right)_\infty$ the underlying $\infty$-category of $\infty$-categories. 

\begin{pro}\label{p:ses}
Let $\M$ be a model category and $\F: \M \lrar \ModCat$ a proper relative functor. Consider the integral model structure on $\int_{\M}\F$ (see Theorem~\ref{model structure}). Let $(A,X),(B,Y) \in \int_{\M}\F$ and consider a map $f: A \lrar B$. Then the sequence
$$ \Map^{\der}_{\F(B)}(f_!X,Y) \lrar \Map^{\der}((A,X),(B,Y)) \lrar \Map^{\der}_{\M}(A,B) $$
is a homotopy fibration sequence (with respect to the base point $f \in \Map^{\der}_{\M}(A,B)$).
\end{pro}
\begin{proof}
According to~\cite[Proposition 3.10]{HP}, there exists an equivalence of $\infty$-categories over $\M_\infty$
\begin{equation}\label{e:hinich}
\xymatrix{
\l(\int_{\M}\F\r)_{\infty} \ar_{\pi}[dr]\ar^{\simeq}[rr] && \int_{\M_\infty}\F_\infty \ar^{p}[dl] \\
& \M_\infty & \\
}
\end{equation}
where $\int_{\M_\infty}\F_\infty \lrar \M_\infty$ is the coCartesian fibration classifying the functor $\F_\infty:\M_\infty \lrar \Cat_\infty$ underlying $\F$. Applying the dual statement of~\cite[Proposition 2.4.4.3]{Lur09} to the $p$-coCartesian morphism $(f,\Id):(A,X) \lrar (B,f_!X)$ we may deduce that the square
$$ \xymatrix{
\Map^h((B,f_!X),(B,Y)) \ar[r]\ar^{q}[d] & \Map^h((A,X),(B,Y)) \ar[d] \\
\Map^h_{\M}(B,B) \ar[r] & \Map_{\M}(A,B) \\
}$$
is homotopy Cartesian. In light of~\ref{e:hinich} we can identify the homotopy fiber of $q$ over $\Id \in \Map^{\der}_{\M}(B,B)$ with the derived mapping space $\Map^{\der}_{\F(B)}(f_!X,Y)$. We hence obtain a homotopy Cartesian square
$$ \xymatrix{
\Map^{\der}_{\F(B)}(f_!X,Y) \ar[r]\ar[d] & \Map^{\der}((A,X),(B,Y)) \ar[d] \\
\{\Id\} \ar^{f}[r] & \Map^{\der}_{\M}(A,B) \\
}$$
The desired result now follows.
\end{proof}

We shall now describe the behaviour of $n$-truncation in the integral model structure.
\begin{pro}\label{p:trunct-int}
Let $\M$ be a model category and $\F: \M \lrar \ModCat$ a proper relative functor. Consider the integral model structure on $\int_\M\F$. Then
\begin{enumerate}
\item
An object $(A,X) \in \int_\M\F$ is $n$-truncated if and only if $A$ is $n$-truncated in $\M$ and $X$ is $n$-truncated in $\F(A)$.
\item
Let $(f,\vphi):(A,X) \lrar (B,Y)$ be a map in $\int_\M\F$ such that $(B,Y)$ is $n$-truncated. Assume that $f: A \lrar B$ is an $n$-truncation in $\M$ and that $\vphi: f_!X \lrar Y$ is an $n$-truncation in $\F(B)$. Then $(f,\vphi)$ is an $n$-truncation in $\int_M\F$.
\end{enumerate}
\end{pro}
\begin{proof}\
\begin{enumerate}

\item
First assume that $(A,X)$ is $n$-truncated. Since $\pi: \int_\M\F \lrar \M$ is in particular a right Quillen functor (see~\cite[Corollary 5.8]{HP}), the image $A = \pi(A,X)$ is $n$-truncated in $\M$. To see that $X$ is $n$-truncated in $\F(A)$ consider another object $X' \in \F(A)$. According to Proposition~\ref{p:ses} we have a homotopy fibration sequence
$$ \Map^{\der}_{\F(A)}(X',X) \lrar \Map^{\der}((A,X'),(A,X)) \lrar \Map^{\der}_{\M}(A,A) $$
with respect to the base point $\Id \in \Map^{\der}_{\M}(A,A)$. Since $(A,X)$ is $n$-truncated and $A$ is $n$-truncated in $\M$ we conclude that $\Map^{\der}_{\F(A)}(X',X)$ is an $n$-truncated space for every $X' \in \F(A)$. This means that $X$ is $n$-truncated in $\F(A)$.

Now assume that $A$ is $n$-truncated in $\M$ and $X$ is $n$-truncated in $\F(A)$. Let $(A',X')$ be an object in $\int_\M\F$. For each $f: A' \lrar A$, consider the homotopy fibration sequence
$$ \Map^{\der}_{\F(A)}(f_!X',X) \lrar \Map^{\der}((A',X'),(A,X)) \lrar \Map^{\der}_{\M}(A',A) $$
given by Proposition~\ref{p:ses}. By our assumptions both $\Map^{\der}_{\M}(A',A)$ and \\ $\Map^{\der}_{\F(A)}(f_!X',X)$ are $n$-truncated and so $\Map^{\der}((A',X'),(A,X))$ is $n$-truncated as well. This shows that $(A,X)$ is an $n$-truncated object of $\int_\M\F$.

\item
Let $(C,Z)$ be an $n$-truncated object in $\int_\M\F$. According to $(1)$ we see that $C$ is $n$-truncated in $\M$ and $Z$ is $n$-truncated in $\F(C)$. Now for every map $g: B \lrar C$ we obtain a homotopy commutative diagram of the form

$$
\xymatrix@=15pt{
\Map^{\der}_{\F(C)}(Y,g^*Z) \ar[r]^{\simeq}\ar[d]_{\simeq}^{(-)\circ \vphi} & \Map^{\der}_{\F(C)}(g_!Y,Z) \ar[r]\ar^{(-)\circ (g_!\vphi)}[d] & \Map^{\der}((B,Y),(C,Z)) \ar[r]\ar^{(-)\circ (f,\vphi)}[d] & \Map^{\der}_{\M}(B,C)\ar^{(-)\circ f}_{\simeq}[d] \\
\Map^{\der}_{\F(C)}(f_!X,g^*Z) \ar[r]^{\simeq} & \Map^{\der}_{\F(C)}(g_!f_!X,Z) \ar[r] & \Map^{\der}((A,X),(C,Z)) \ar[r] & \Map^{\der}_{\M}(A,C) \\
}$$
where the left-most horizontal arrows are the equivalences induced by the Quillen adjunction $g_! \dashv g^*$, and the middle and right-most squares comprise the map of homotopy fibration sequences given by Proposition~\ref{p:ses}. Now the right vertical map is an equivalence because $f$ is an $n$-truncation, the left vertical map is an equivalence because $\vphi$ is an $n$-truncation. It hence follows that the induced map
$$ (-)\circ (f,\vphi): \Map^{\der}((B,Y),(C,Z)) \lrar \Map^{\der}((A,X),(C,Z)) $$
is an equivalence as desired.
\end{enumerate}
\end{proof}

\section{Truncation in the integral model category of Segal group actions}\label{ss:tr-seg}

In this section we will describe $n$-truncation functors in the model categories of Segal groups and Segal group actions. Our main goal is to prove the following (see Theorem~\ref{t:act-truncation} and Theorem~\ref{t:converge} below):
\begin{itemize}
\item
The natural map of Segal group actions $$ \xymatrix{
X_\bullet \ar[r]\ar[d] & \P_n(X_\bullet) \ar[d] \\
A_\bullet \ar[r] & \P_n(A_\bullet) \\
}$$ is an $n$-truncation map, where $\P_n$ is the Postnikov piece functor of spaces applied level-wise.
\item
The resulting \textbf{Postnikov tower} of Segal group actions $\{\P_n(X_\bullet) \lrar \P_n(A_\bullet)\}$ converges to $X_\bullet \lrar A_\bullet$.
\end{itemize}

Let us begin by examining truncation in the slice category of spaces. 

\begin{notn}
Let $f:X \lrar A$ be a map in $\SS$. We will denote by 
$$ X \x{\simeq}{\lrar} \ovl{X} \x{f^{\fib}}{\lrar} A $$
the functorial factorization of $f$ into a trivial cofibration followed by a fibration.
\end{notn}

We now recall the following construction, taken from \cite[2.4]{DK}.
\begin{define}
Let $f: X \lrar A$ be a map in $\SS$. For $n\geq 1$ we will denote by
$\cosk_n(f)$
the simplicial set whose $k$-simplices are the set of commutative squares of the form
$$ \xymatrix{
\sk_n(\Del^k) \ar[r]\ar[d] & X \ar[d] \\
\Del^k \ar[r] & A. \\
}$$
We have canonical maps 
$$ \xymatrix{
X \ar[rr]\ar_{f}[dr] && \cosk_n(f) \ar^{f_n}[dl] \\
& A & \\
}$$
We then define the \textbf{relative Postnikov $n$-piece of $X$ over $A$} to be
$$ \P_n(X/A) \x{\df}{=} \cosk_{n+1}\left(f^{\fib}\right) $$
\end{define}

\begin{rem}\label{r:relative}
It is straightforward to verify that if $f: X \lrar A$ is a fibration then $f_n:\cosk_n(f) \lrar A$ is a fibration as well. Note that for $A = *$ the above construction reproduces one of the standard models for the Postnikov $n$-piece of $X$, namely, the coskeleton of its Kan replacement. For general $A$ we see that $\P_n(X/A)$ is $n$-truncated as an object in $\SS_{/A}$ and the map $X \lrar \P_n(X/A)$ is an $n$-truncation in $\SS_{/A}$. Furthermore, for any $f: X \lrar A$ there exists an equivalence 
$$ \xymatrix{
\P_n(X) \ar^{\simeq}[rr]\ar[dr] && \P_n(X/\P_n(A)) \ar[dl] \\
& \P_n(A) & \\
}$$
in $\SS_{/\P_n(A)}$.
\end{rem}


Let $\I$ be a small category and consider the functor category $\SS^{\I}$ equipped with either the injective or the projective model structure. Then an object $X \in \SS^{\I}$ is $n$-truncated if and only if $X(i)$ is $n$-truncated for every $i \in \I$. This follows from the fact that every object in $\SS^{\I}$ is a homotopy colimit of representables $R_i \in \SS^{\I}$ (see~\cite[Proposition 2.9]{Dug}) in conjunction with
$$ \Map^{\der}_{\SS^{\I}}(R_i,X) \simeq X(i) $$
Similarly, if $A \in \SS^{\I}$ is any object then we can consider the slice model category $\SS^{\I}_{/A}$. As in example~\ref{e:slice} we see that an object $\vphi: X \lrar A$ in $\SS^{\I}_{/A}$ is $n$-truncated if the homotopy fiber of $\vphi(i):X(i) \lrar A(i)$ over every point of $A(i)$ is $n$-truncated (equivalently, if $X(i) \lrar A(i)$ is $n$-truncated in $\SS_{/A(i)}$ for every $i \in \I$).

The case of $\SS^{\I}_{/A}$ considered above is in fact extremely well-behaved. Given an object $X \lrar A$ in $\SS^{\I}_{/A}$ one can define $\P_n(X/A)$ by setting
$$ \P_n(X/A)(i) = \P_n(X(i)/A(i)) .$$
Using the same argument as in~\cite{Bie} one can verify that the functor $Q = \P_n(-/A)$ satisfies Bousfield's axioms $\mathbf{A}.4, \mathbf{A}.5$ and $\mathbf{A}.6$ (see~\cite[Definition 2.2]{Bie}) and hence one can left-Bousfield localize the slice injective model structure on $\SS^{\I}_{/A}$ so that the new weak equivalences are the maps 
$$ \xymatrix{ 
X \ar[rr]\ar[dr] && Y \ar[dl] \\
& A & \\
}$$
such that
$$ \xymatrix{ 
\P_n(X/A) \ar[rr]\ar[dr] && \P_n(Y/A) \ar[dl] \\
& A & \\
}$$
is a weak equivalence in $\SS^{\I}_{/A}$. Furthermore, the new fibrant objects are precisely the old fibrant objects which are furthermore $n$-truncated. This means in particular that for any $X \in \SS^{\I}_{/A}$ the natural map
$$ X \lrar \P_n(X/A) $$
is an $n$-truncation in the sense of Definition~\ref{d:truncated}.

Our next goal in this subsection is to construct truncation maps for the model categories $\Seg$ of Segal groups and the model category 
$$ \displaystyle\mathop{\int}_{A_\bullet \in \Seg} \left(\sS_{/\iota(A_\bullet)}\right)_{\Act} $$ 
of all Segal groups actions. We begin with a general lemma which allows one to study $n$-truncation in a model category $\M$ via a suitable Quillen adjunction to another model category $\N$.
\begin{lem}\label{l:general}
Let 
$$ \xymatrix@=13pt{
\M \ar[rr]<1ex>^(0.5){\L} && \N \ar[ll]<1ex>_(0.5){\upvdash}^(0.5){\R}
}$$
be a Quillen adjunction such that the derived counit map
$$ \L((\R(X))^{\cof}) \lrar X $$ 
is a weak equivalence for every fibrant $X \in \N$. Then
\begin{enumerate}
\item
An object $X \in \N$ is $n$-truncated if and only if $R\left(X^{\fib}\right)$ is $n$-truncated in $\M$.
\item
Let $f: X \lrar Y$ be a map in $\N$ where $Y$ is $n$-truncated. Assume that the induced map $\R\left(X^{\fib}\right) \lrar \R\left(Y^{\fib}\right)$ is an $n$-truncation. Then $f$ is an $n$-truncation.
\end{enumerate}
\end{lem}
\begin{proof}
Let us begin with $(1)$. First by adjunction it follows that the derived functor $\R((-)^{\fib})$ sends $n$-truncated objects to $n$-truncated objects. On the other hand, under the assumptions of the lemma the derived functor $\R((-)^{\fib})$ induces an equivalence on derived mapping spaces
$$ \Map^{\der}_\N(Y,X) \x{\simeq}{\lrar} \Map^{\der}_\M(\R\left(Y^{\fib}\right), \R\left(X^{\fib}\right)) $$
for every $Y,X \in \N$. It hence follows that if $\R\left(X^{\fib}\right)$ is $n$-truncated then $X$ is $n$-truncated. 

We now proceed to prove $(2)$. Let $Z \in \N$ be an $n$-truncated object. By the above we get that $\R(Z^{\fib})$ is $n$-truncated as well. Furthermore we know that
$$ \Map^{\der}(Y,Z) \simeq \Map^{\der}\left(\R\left(Y^{\fib}\right),\R\left(Z^{\fib}\right)\right) $$
and
$$ \Map^{\der}(X,Z) \simeq \Map^{\der}\left(\R\left(X^{\fib}\right),\R\left(Z^{\fib}\right)\right) $$
Since we assumed that the map $\R\left(X^{\fib}\right) \lrar \R\left(Y^{\fib}\right)$ is an $n$-truncation it follows that $f: X \lrar Y$ is an $n$-truncation as well.
\end{proof}

\begin{cor}\label{c:before}
Let $X_\bullet$ be a reduced simplicial space. Then 
\begin{enumerate}
\item
$X_\bullet$ is $n$-truncated in $\sS_0$ if and only if $\iota(X_\bullet)$ is $n$-truncated in $\\sS$, i.e., if and only if $X_k$ is an $n$-truncated space for every $k$.
\item
The natural map $X_\bullet \lrar \P_n(X_\bullet)$ is an $n$-truncation in $\sS_0$.
\end{enumerate}
\end{cor}
\begin{proof}
Apply Lemma~\ref{l:general} to the Quillen adjunction
$$ \xymatrix@=13pt{
\sS \ar[rr]<1ex>^(0.5){(-)^{\red}} && \sS_0 \ar[ll]<1ex>_(0.5){\upvdash}^(0.5){\iota}
}$$

\end{proof}

\begin{cor}\label{c:trunct-seg}
Let $A_\bullet$ be a reduced simplicial space. Then 
\begin{enumerate}
\item
$A_\bullet$ is $n$-truncated in $\Seg$ if and only if $A^{\fib}_\bullet$ is level-wise $n$-truncated.
\item
The natural map $A_\bullet \lrar \P_n\left((A_\bullet)^{\fib}\right)$ is an $n$-truncation in $\Seg$.
\end{enumerate}
\end{cor}
\begin{proof}
Apply Lemma~\ref{l:general} to the Quillen adjunction
$$ \xymatrix@=13pt{
\sS_0 \ar[rr]<1ex>^(0.5){\Id} && \Seg \ar[ll]<1ex>_(0.5){\upvdash}^(0.5){\Id}
}$$
\end{proof}

\begin{rem}\label{r:seg-up-to}
Since $\P_n$ preserves Cartesian products up to weak equivalence we see that for any Segal group $A_\bullet$ the $n$-truncated object $\P_n(A_\bullet)$ is a Segal group up to level-wise weak equivalence.
\end{rem}

\begin{cor}\label{c:trunct-act}
Let $A_\bullet$ be an object in $\Seg$ and let $f:X_\bullet \lrar \iota(A_\bullet)$ be an object in $\l(\sS_{/\iota(A_\bullet)}\r)_{\Act}$. Then
\begin{enumerate}
\item
$f$ is $n$-truncated in $\l(\sS_{/\iota(A_\bullet)}\r)_{\Act}$ if and only if its fibrant replacement $f^{fib}$ in $\l(\sS_{/\iota(A_\bullet)}\r)_{\Act}$ is $n$-truncated when considered as an object in the slice Reedy model structure $\sS_{/\iota(A_\bullet)}$.
\item
The natural map
$$ \xymatrix{
X_\bullet \ar[rr]\ar[dr] && \P_n\l(f^{fib}\r) \ar[dl] \\
& \iota(A_\bullet) & \\
}$$
is an $n$-truncation in $\l(\sS_{/\iota(A_\bullet)}\r)_{\Act}$.
\end{enumerate}
\end{cor}
\begin{proof}
Apply Lemma~\ref{l:general} to the Quillen adjunction
$$ \xymatrix@=13pt{
\sS_{/\iota(A_\bullet)} \ar[rr]<1ex>^(0.5){\Id} && \l(\sS_{/\iota(A_\bullet)}\r)_{\Act} \ar[ll]<1ex>_(0.5){\upvdash}^(0.5){\Id}
}$$
\end{proof}

\begin{thm}\label{t:act-truncation}
Let $A_\bullet$ be a Segal group and $p:X_\bullet \lrar A_\bullet$ a Segal group action. Then the map in $\displaystyle\mathop{\int}_{A_\bullet \in \Seg} \l(\SS_{/\iota(A_\bullet)}\r)_{\Act}$ determined by the diagram
$$ \xymatrix{
X_\bullet \ar[r]\ar[d] & \P_n(X_\bullet) \ar[d] \\
A_\bullet \ar[r]_(0.45){\tau_n} & \P_n(A_\bullet) \\
}$$
is an $n$-truncation.
\end{thm}
\begin{proof}
According to Corollary~\ref{c:trunct-seg}, the map $\tau_n:A_\bullet \lrar \P_n(A_\bullet)$ is an $n$-truncation in $\Seg$. The object $(\tau_n)_!(X_\bullet) \in \l(\SS_{/\iota(A_\bullet)}\r)_{\Act}$ is simply given by the composed map
$$ \tau_n \circ p: X_\bullet \lrar \P_n(A_\bullet) $$
Let 
$$ q: Y_\bullet \lrar \P_n(A_\bullet) $$
be a fibrant replacement of $\tau_n \circ p$ in $\l(\SS_{/\iota(A_\bullet)}\r)_{\Act}$, so that we have a commutative diagram
$$ \xymatrix{
X_\bullet \ar[d]\ar@{=}[r] & X_\bullet \ar[r]\ar[d] & Y_\bullet \ar[d] \\
A_\bullet \ar[r] &\P_n(A_\bullet) \ar@{=}[r] & \P_n(A_\bullet) \\
}$$
Combining Corollary~\ref{c:trunct-act} and Proposition~\ref{p:trunct-int} we see that the square
\begin{equation}\label{e:postnikov} 
\xymatrix{
X_\bullet \ar[r]\ar[d] & \P_n(Y_\bullet/\P_n(A_\bullet)) \ar[d] \\
A_\bullet \ar[r] & \P_n(A_\bullet) \\
}
\end{equation}
determines an $n$-truncation map in $\displaystyle\mathop{\int}_{A_\bullet \in \Seg} \l(\SS_{/\iota(A_\bullet)}\r)_{\Act}$. 

In light of Remark~\ref{r:relative} we know that $\P_n(Y_\bullet/\P_n(A_\bullet)) \simeq \P_n(Y_\bullet)$ over $\P_n(A_\bullet)$. Hence in order to prove the theorem we need to show that the map 
$$ \P_n(X_\bullet) \lrar \P_n(Y_\bullet) $$
is an equivalence of simplicial spaces. In other words, we need to show that the maps
$$ X_k \lrar Y_k $$
induce an isomorphism on homotopy groups up to dimension $n$. Since both $X_\bullet \lrar A_\bullet$ and $Y_\bullet \lrar \P_n(A_\bullet)$ are Segal group actions and since the map $A_\bullet \lrar \P_n(A_\bullet)$ is an isomorphism on homotopy groups up to dimension $n$ we see that it will be enough to show that the map
$$ X_0 \lrar Y_0 $$
induces an isomorphism on homotopy groups up to dimension $n$.

Since the map $X_\bullet \lrar Y_\bullet$ is a weak equivalence in $\sS_{/\iota(\P_n(A_\bullet))}$ we see that the induced map
$$ |X_\bullet| \lrar |Y_\bullet| $$
is a weak equivalence over $|\P_n(A_\bullet)|$, and hence the homotopy fiber $Y_0'$ of $|X_\bullet| \lrar |\P_n(A_\bullet)|$ is naturally equivalent to $Y_0$. It will hence suffice to show that the natural map
$$ X_0 \lrar Y_0' $$
induces an isomorphism on homotopy groups up to dimension $n$.

Let $\WW_n(A_\bullet) \lrar A_\bullet$ be the homotopy fiber of the map $A_\bullet \lrar \P_n(A_\bullet)$ in $\Seg$ (this coincides with the level-wise homotopy fiber of the corresponding map of reduced simplicial spaces, since both $A_\bullet$ and $\P_n(A_\bullet)$ are Segal groups up to Reedy fibrancy). Since the realization functor $|-|: \Seg \lrar \SS_0$ is a left Quillen equivalence we see that the sequence
$$ |\WW_n(A_\bullet)| \lrar |A_\bullet| \lrar |\P_n(A_\bullet)| $$
is again a homotopy fibration sequence. Now consider the diagram
$$ \xymatrix{
X_0 \ar[r]\ar[d] & Y_0' \ar[r]\ar[d] & |W_n(A_\bullet)| \ar[d] \\
|X_\bullet| \ar[d]\ar@{=}[r] & |X_\bullet| \ar[d]\ar[r] & |A_\bullet| \ar[d] \\
|A_\bullet| \ar[r] & |\P_n(A_\bullet)| \ar@{=}[r] & |\P_n(A_\bullet)| \\
}
$$
where the right top horizontal map is the induced map on homotopy fibers. According to Lemma~\ref{l:wierd} below, the sequence
$$ X_0 \lrar Y_0' \lrar |\WW_n(A_\bullet)| $$
is a homotopy fibration sequence. Since $\WW_n(A_\bullet)$ is levelwise $n$-connected (i.e., has no homotopy groups in dimension $\leq n$), and is a Segal group up to Reedy fibrancy, we conclude that $|\WW_n(A_\bullet)|$ is $(n+1)$-connected (see~\cite[Proposition 1.5(a)]{Seg}). The desired result now follows from the long exact sequence in homotopy groups.

\end{proof}

\begin{lem}\label{l:wierd}
Consider a commutative diagram of pointed spaces of the form
$$ \xymatrix{
X \ar^{f}[d]\ar@{=}[r] & X \ar[d]\ar[r] & A \ar^{g}[d] \\
A \ar^{g}[r] & P \ar@{=}[r] & P \\
}
$$
Then the sequence
$$ C \lrar D \lrar E $$
obtained by passing to homotopy fibers of the vertical maps is itself a homotopy fibration sequence.
\end{lem}
\begin{proof}
We may assume without loss of generality that the maps $f:X \lrar A$ and $g:A \lrar P$ are Kan fibrations (which means that $h = g \circ f$ is a Kan fibration as well), so that the homotopy fibers $C,D$ coincide with the actual fibers $C = f^{-1}(*), D = h^{-1}$ and $E = g^{-1}(*)$. It is then clear that the natural map
$$ D \lrar E \times_A X $$
is an isomorphism. Furthermore, since the map $X \lrar A$ is a fibration we see that the fiber product on the right hand side coincides with the associated homotopy fiber product $E \times^h_A X$. It is then clear that the resulting sequence
$$ C \simeq *\times^h_A X \lrar E \times^h_A X \lrar E $$
is a homotopy fibration sequence.
\end{proof}

\begin{rem}\label{r:act-up-to}
Since $\P_n$ preserves Cartesian products up to weak equivalence we see that for any Segal group action $f:X_\bullet \lrar A_\bullet$ the $n$-truncated object $\P_n(X_\bullet) \lrar \P_n(A_\bullet)$ is a Segal group action up to levelwise weak equivalence (see Remark~\ref{r:seg-up-to}).
\end{rem}

\subsection{Convergence of the Postnikov tower}\label{ss:conv}


Given a simplicial group $G$  we will denote by $G \racts X$ a space $X$ equipped with a strict action of $G$. Similarly, given a Segal group $A_\bullet$ we will denote a Segal group action of the form
$X_\bullet \lrar A_\bullet$ by $A_1\overset{h}{\racts} X_0$. The latter notation comes from viewing a Segal group action as coherent action of the loop space $A_1\simeq \Om |A_\bullet|$ on the space $X_0$. 
Recall the Quillen equivalence
$$
\xymatrix{\Lambda^L: \displaystyle\mathop{\int}_{A_\bullet \in \l(\sS_0\r)_{\segal}}\l(\sS_{/\iota(X_\bullet )}\r)_{\act} \ar[r]<1ex> & \displaystyle\mathop{\int}_{G \in \sGr}\SS^{\B G}:\Lambda^R\ar[l]<1ex>_(0.4){\upvdash}}
$$
established in Corollary~\ref{c:equivalence-3}. For a $G$-space $X$ we will denote its corresponding Segal group action by 
$$ B(G,X)_\bullet\lrar B(G)_\bullet \x{\df}{=} \mathbb{R}\Lambda^R(G,X).$$


Now suppose $X$ is a $G$-space and consider its Segal group action 
$$B(G,X)_\bullet\lrar B(G)_\bullet.$$ 
Using Theorem~\ref{t:act-truncation} and Remark~\ref{r:act-up-to} we obtain a tower of truncated Segal group actions

\begin{equation}\label{e:segal postnikov} 
\xymatrix{ \\& \x{\hspace{ 20pt}\vdots}{^{P_nG \x{\text{\tiny h}}{\racts} }P_nX}\ar[d]<2ex>\\&\hspace{ 19pt}\vdots\ar[d]<2ex>\\&^{P_1G \x{\text{\tiny h}}{\racts} }P_1X\ar[d]<2ex>\\  ^{G\x{\text{\tiny h}}{\racts}}X\ar[r]<-1.5ex>\ar[ur]<-1.5ex>\ar[uuur]<-1.5ex> & ^{P_0G \x{\text{\tiny h}}{\racts} }P_0X}
\end{equation}
in which the maps are maps between Segal group actions.



Let us first observe that the Quillen equivalence of Corollary \ref{c:equivalence-3} induces a Quillen equivalence

$$ \xymatrix@=13pt{
\left(\displaystyle\mathop{\int}_{A_\bullet \in \sS_0} \sS_{/\iota(A_\bullet)} \right)^{\NN^{\op}}\ar[rr]<1.2ex>_(0.54){\upvdash}  && 
\left(\displaystyle\mathop{\int}_{G \in \sGr} \SS^{\B G}\right)^{\NN^{\op}} \ar[ll]<1.2ex> \\
}
$$
where $\NN$ is the poset of natural numbers and where the model structure we use on the diagram categories is the Reedy model structure. Applying the (derived) left Quillen functor above on the tower \ref{e:segal postnikov}, will now yield a tower of (ordinary) group actions 

\begin{equation}\label{e:borel postnikov} 
\xymatrix{ \\& \x{\hspace{ 20pt}\vdots}{^{G_n \racts }X_n}\ar[d]<2ex>\\&\hspace{ 19pt}\vdots\ar[d]<2ex>\\&^{G_1 \racts }X_1\ar[d]<2ex>\\  ^{G\racts}X\ar[r]<-1.5ex>\ar[ur]<-1.5ex>\ar[uuur]<-1.5ex> & ^{G_0 \racts }X_0}
\end{equation}
where $G_n\simeq \P_nG$ and $X_n\simeq \P_nX$.
This means that the coherent tower \ref{e:segal postnikov} and the strict tower \ref{e:borel postnikov} contain the same homotopy-theoretical information. 

Recall that the tower \ref{e:segal postnikov} can be built for any Segal group action $X_\bullet \lrar A_\bullet$ without reference to a strict group action $G\racts X$. 
We now claim that
\begin{thm}\label{t:converge}
For any Segal group action $X_\bullet \lrar A_\bullet$, viewed as an object in $$\displaystyle\mathop{\int}_{A_\bullet \in \l(\sS_0\r)_{\segal}} \l(\sS_{/\iota(A_\bullet)}\r)_{\act},$$ the tower of \ref{e:segal postnikov} converges in that $$(X_\bullet\lrar A_\bullet)\simeq \holim_n  (\P_nX_\bullet\lrar \P_nA_\bullet).$$   
\end{thm}

\begin{proof}
Since each $\P_n(A_\bullet)$ is a Segal group up to weak equivalence (see Remark~\ref{r:seg-up-to}), we can compute the homotopy limit of $\{\P_n(A_\bullet)\}$ separately in each simplicial degree, yielding

$$ \holim_n \P_nA_\bullet\simeq A_\bullet $$ 
in $\l(\sS_0\r)_{\segal}$. Thus, in order to compute $\holim_n  (\P_nX_\bullet\lrar \P_nA_\bullet)$ in $$\displaystyle\mathop{\int}_{A_\bullet \in \l(\sS_0\r)_{\segal}} \l(\sS_{/\iota A_\bullet}\r)_{\act}$$ we first pull back each $$\P_nX_\bullet\lrar \P_nA_\bullet$$ to the fiber $\sS_{/\iota(A_\bullet)}$ over $A_\bullet$ and compute the homotopy limit there.

For a fixed $n$, the homotopy pullback 
$$\xymatrix{
L_\bullet^{(n)} \ar[r]\ar[d] & \P_nX_\bullet\ar[d]\\ 
A_\bullet\ar[r] & \P_nA_\bullet
}$$  is taken separately in each simplicial degree and by the axioms of a Segal group action we get $$L_k^{(n)}\simeq \P_nX_0\times A_k.$$ We then obtain natural maps 
$$X_\bullet\lrar L_\bullet^{(n)}$$ 
(over $A_\bullet$) inducing a map $X_\bullet\lrar \holim_n L_\bullet^{(n)}$ of Segal group actions over $A_\bullet$. This map is an equivalence in each simplicial degree since $\holim\P_nX_0\simeq X_0$ and hence an equivalence of Segal group actions over $A_\bullet$. The result now follows.
\end{proof}
By the Quillen equivalence of Corollary~\ref{c:equivalence-3}, we get
\begin{cor}
For any simplicial group $G$ and any $G$-space $X$, the tower \ref{e:borel postnikov} converges in that $\holim_n (G_n\racts X_n)\simeq G\racts X$.
\end{cor}

\end{document}